\documentclass[preprint,12pt]{elsarticle}

\usepackage{amssymb}
\usepackage{amsmath}
\usepackage{amsthm}

\usepackage{color}

\newtheorem{thm}{Theorem}
\newtheorem{lem}[thm]{Lemma}

\newtheorem{cor}[thm]{Corollary}

\newtheorem{rmk}[thm]{Remark}
\newtheorem*{notat}{Notation}

\newcommand{\be}{\begin{equation}}
\newcommand{\ee}{\end{equation}}
\newcommand{\bd}{\begin{displaymath}}
\newcommand{\ed}{\end{displaymath}}
\newcommand{\bea}{\begin{eqnarray}}
\newcommand{\eea}{\end{eqnarray}}
\newcommand{\bgt}{\begin{gather}}
\newcommand{\egt}{\end{gather}}

\newcommand{\iu}{\mathrm{i}}

\newcommand{\st}{\,\big|\,}

\newcommand{\id}{I_{\mathrm{d}}}

\newcommand{\N}{\mathbb{N}}

\newcommand{\R}{\mathbb{R}}
\newcommand{\C}{\mathbb{C}}

\newcommand{\cB}{\mathcal{B}}

\newcommand{\cH}{\mathcal{H}}

\newcommand{\cJ}{\mathcal{J}}

\newcommand{\cL}{\mathcal{L}}

\newcommand{\cS}{\mathcal{S}}

\newcommand{\cU}{\mathcal{U}}
\newcommand{\cV}{\mathcal{V}}

\mathchardef\mhyphen="2D

\newcommand{\Real}{{\mathrm{Re}}}
\newcommand{\Imag}{{\mathrm{Im}}}

\DeclareMathOperator*{\slim}{s-lim}

\newcommand{\pot}{q}
\newcommand{\oL}{{L}}
\newcommand{\oLa}{{L^{*}}}
\newcommand{\Lb}{{\Lambda}}
\newcommand{\oLo}{{D_{1}}}
\newcommand{\oLt}{{D_{2}}}
\newcommand{\oLth}{{D_{2}^{1/2}}}
\newcommand{\ioLth}{{D_{2}^{-1/2}}}
\newcommand{\oLot}{{D_{1}D_{2}}}
\newcommand{\oLto}{{D_{2}D_{1}}}

\newcommand{\ures}{\mathbb{U}}
\newcommand{\vres}{\mathbb{V}}

\newcommand{\pchar}{p_{\mathrm{c}}}
\newcommand{\hp}{{h_{\mathrm{p}}}}
\newcommand{\hm}{{h_{\mathrm{m}}}}
\newcommand{\ha}{{h_{\mathrm{a}}}}

\newcommand{\phip}{\phi}
\newcommand{\phim}{\chi}

\newcommand{\ef}{\varphi}
\newcommand{\aef}{\varphi^{*}}

\newcommand{\dom}{\mathrm{dom}}
\newcommand{\ran}{\mathrm{ran}}
\newcommand{\iker}{\mathrm{ker}}

\newcommand{\mur}{\theta}

\journal{Journal of Mathematical Analysis and Applications}

\begin{document}

\begin{frontmatter}



\title{A class of nonselfadjoint spectral differential operators of interest in physics}


\author{Victor Laliena} 

\affiliation{organization={Departamento de Matem\'atica Aplicada and Instituto Universitario de Matem\'aticas y Aplicaciones (IUMA), Universidad de Zaragoza},
            city={Zaragoza},
            country={Spain}}

\begin{abstract}
It is shown that the nonselfadjoint (and non-normal) linear ordinary differential operators of a certain class are spectral operators of scalar type in the sense of Dunford and Bade. Operators of this kind appear in physical problems such as the scattering of spin waves by magnetic solitons.
\end{abstract}


\end{frontmatter}

\section{Introduction \label{sec:intro}}

Besides its intrinsic mathematical interest, spectral analysis of operators is an important tool in physical and engineering problems. Many theoretical results are best interpreted and confronted with experiments through some kind of spectral analysis. From this point of view, theoretical problems involving spectral operators \cite{Dunford1954,Dunford1958} are considerably simplified. Although there are known classes of non-normal spectral differential operators (for instance some ordinary differential operators with purely discrete spectrum \cite{Schwartz1954,Kramer1957,Turner1966}, some ordinary differential operators with periodic coefficients \cite{McGarvey1965}, or some classes of second order elliptic differential operators defined in $L^2(\R^n)$, with $n\geq 3$ \cite{Schwartz1960,Dunford1971}), it is in general difficult to proof that concrete operators are spectral. The root of the difficulties lies in the possible appearance of spectral singularities,  which are absent in the case of normal operators but are typical otherwise \cite{Schwartz1960}.

In this work it is shown that a class of non-normal ordinary differential operators, not considered before, are spectral operators of scalar type in the sense of Dunford \cite{Dunford1954,Dunford1958} (actually, in the sense of Bade, since the operators are unbounded \cite{Bade1954}). These operators appear, for instance, in the theory of scattering of spin waves by one-dimensional solitons \cite{Laliena2021}, and are defined through a differential expression of the form $\oL u=\oLto u$, where $u$ is a function on $\R$ and
\be
D_i = -D^2 + \pot_i + h_i, \quad (i=1,2). \label{eq:def_L_i}
\ee
In the above expression $D$ represents the derivative, $Du=u^\prime$, $h_2>h_1>0$ are constants, and $\pot_1$ and $\pot_2$ are real functions defined on $\R$. We require, for $i=1,2$,
\begin{enumerate}
\item $\pot_i\in C^1(\R)$, $\pot_i^\prime\in AC(\R)$, $\pot_i^{\prime\prime}\in L^2(\R)$,
\item $\int_\R (1+|x|^3)|\pot^{(k)}_i(x)|dx<\infty$, $0\leq k\leq 2$,
\end{enumerate}
where $AC(\R)$ is the set of complex functions on $\R$ which are absolutely continuous on any compact interval, and the primes and the superscript $(k)$ on a function denote, respectively, the derivatives and the $k$-th derivative of the function.  

To avoid symbol proliferation, we use the same symbols, $\oLo$, $\oLt$, $\oL$, $\oLa$ (the last is defined below), for the differential expressions and for the operators defined by the differential expressions on appropriate domains. Since in this work each differential expression is associated with a unique operator, there is no ambiguity in this respect. It should be clear from the context whether the symbol represents the operator or the differential expression applied to some suitable function.

With the above assumptions, $\oLo$ and $\oLt$ are selfadjoint operators in $L^2(\R)$, with domain $H^2(\R)$, and with spectra bounded from below. For $i=1,2$, the essential spectrum of $D_i$ is $[h_i,\infty)$, and the point spectrum is contained in $(-\infty,h_i]$.  We require further that $\oLt$ be positive definite and have a bounded inverse. That is, denoting the scalar product and the norm in $L^2(\R)$ by $(\cdot,\cdot)$ and $\|\cdot\|$, respectively, there is $c_2>0$ such that $(\oLt u,u)\geq c_2\|u\|^2$ for any $u\in\dom(\oLt)$. 

It may be that $\iker(\oLo)\neq\{0\}$ (this happens generically in applications to spin wave dynamics). The orthogonal projection onto $\iker(\oLo)$, denoted by $P_0$, is a finite rank operator of rank at most two. Let $\id$ be the identity operator and define $P_R=\id-P_0$. The restriction of $\oLo$ to the subspace $\dom(\oLo)\cap\,\ran(P_R)$, denoted by $D_{1R}$, is injective and has a bounded inverse, $D_{1R}^{-1}:\ran(\oLo)\to\ran(P_R)$, and the operator $\oLo+P_0$ is injective and has a bounded inverse given by $(\oLo+P_0)^{-1}=D_{1R}^{-1}P_R+P_0$.

From $\oLo$ and $\oLt$ we build the linear operator $\oL:\dom(\oL)\to L^2(\R)$ as
\be
\begin{gathered}
\dom(\oL) = \big\{u\in\dom(\oLo) \st \oLo u\in \dom(\oLt)\}, \\[4pt]
\oL u=\oLto u, \quad u\in\dom(\oL).
\end{gathered}
\ee
The domain of $\oL$ is dense in $L^2(\R)$ since it contains $\cS(\R)$, the Schwartz space of rapidly decreasing functions. 
By elementary means it is proved that $\oLa$, the adjoint of $\oL$, is given by
\be
\begin{gathered}
\dom(\oLa) = \big\{u\in\dom(\oLt) \st \oLt u\in \dom(\oLo)\}, \\[4pt]
\oLa u=\oLot u, \quad u\in\dom(\oLa).
\end{gathered}
\ee
Again $\cS(\R)\subset\dom(\oLa)$.
By similar means it is proved that $\oL^{**}=\oL$, and therefore, $\oL$ and $\oLa$ are closed operators. They are not selfadjoint unless $\pot_1=\pot_2$, since, on $\dom(\oL)\cap\dom(\oLa)$,
\be
\oLa - \oL = 2(\pot_1^\prime-\pot_2^\prime)D +\pot_1^{\prime\prime}-\pot_2^{\prime\prime}.
\ee
We notice that $\oLa$ is a relatively bounded perturbation of $\oL$, with zero relative bound.

Let us remark two facts that will be used in the proof of some results. First,
 if $f\in\cS(\R)$ then $\oL f$ and $\oLa f$ are in $L^1(\R)$. Second, $\iker(\oL)=\iker(\oLo)$ and if $f\in\iker(\oLo)$  then $f\in\dom(\oLa)$, because in this case $\oLt f=Bf$, where $B$ is the operator of multiplication by $h_2-h_1+\pot_2-\pot_1$, and then $Bf\in H^2(\R)$. Therefore, $\ran(P_0)\subset\dom(\oLa)$.

\begin{notat}
\textnormal{
We denote by primes the derivatives with respect to the \textit{first variable} of any function. 
The $k$-th derivative with respect to the \textit{first variable} of a function is denoted by the superscript $(k)$. 
For a complex number $z$ the symbols $z^{1/2}=\sqrt{z}$ represent the principal branch of the square root of $z$. \textit{Almost all}, abbreviated \textit{a.a.}, and almost everywhere (\textit{a.e.}) are to be understood with respect to the Lebesgue measure.
}

\textnormal{
By $\rho(A)$, $\sigma(A)$, $\sigma_p(A)$, $\sigma_r(A)$, and $\sigma_c(A)$, we denote, respectively, the resolvent set, the spectrum, the point spectrum, the residual spectrum, and the continuous spectrum of the operator $A$.
If $A$ is selfadjoint, its essential spectrum is denoted by $\sigma_e(A)$. Otherwise, the five subsets of the spectrum defined in chapter 9 of Edmunds and Evans \cite{Edmunds2018}, all of which are occasionally called essential spectrum, are denoted by $\sigma_{ek}(A)$, $1\leq k\leq 5$. 
For $z\in\rho(A)$ the resolvent of $A$ is defined by $R_A(z)=(A-z\id)^{-1}$. 
}

\textnormal{
We use the symbol $\id$ for any identity operator. It should be clear from the context in which space acts this identity operator. To lighten the notation we introduce the symbols $\cH=L^2(\R)$ and $\cS=\cS(\R)$. As already said, we denote the scalar product and the norm in $\cH$ by $(\cdot,\cdot)$ and $\|\cdot\|$, respectively. 
}
\end{notat}

\section{Summary of main results \label{sec:results}}

The main results of this work are easily summarized: there is a spectral resolution of the identity for $\oL$ (theorem \ref{thm:spectral_measure}), and $\oL$ is a spectral operator of scalar type in the sense of Dunford and Bade (theorem \ref{thm:spectral_operator}). The same statements hold for $\oLa$.
Theorems \ref{thm:spectral_measure} and \ref{thm:spectral_operator} follow from the three key results described below.
 
\textit{1.} The spectral properties of $\oL$ and $\oLa$ are related to those of the selfadjoint operator $\Lb =\oLth\oLo\oLth$. This fact allows to show that the spectra of $\oL$ and $\oLa$ lie on the real axis and to obtain a homomorphism $B_{\oL}$ between the Boolean algebra of bounded Borel sets of $\R$ and a Boolean algebra of projections contained in $\cL(\cH)$, the elements of which reduce $\oL$ (the adjoint projections reduce $\oLa$). The projections  $B_{\oL}(b)$ are uniformly bounded when restricted to the Borel  sets $b$ contained of any bounded subset of $\R$, what implies that  $\oL$ and $\oLa$ do not have spectral singularities. These ideas are developed in sections \ref{sec:spectrum} and \ref{sec:projections_bounded}.
From these results it remains unclear whether the homomorphism is bounded and whether it can be extended to a homomorphism between the Boolean algebra of Borel set of $\R$ and a Boolean algebra of projections contained in $\cL(\cH)$. 

\textit{2.} There is a spectral expansion associated to $\oL$ (an expansion in terms of the bounded solutions of $\oL u=\lambda u$, or ``generalized eigenfunctions'', see section \ref{sec:expansion}), which can be obtained after a thorough analysis of the resolvent of $\oL$, exploiting the fact that $\oL$ is an ordinary differential operator (part of this analytic work was made long ago by Kemp \cite{Kemp1960} for a more general class of ordinary differential operators which include $\oL$ and $\oLa$, but his results have to be sharpened for this work). The analysis of the Green function is carried out in section \ref{sec:green}, and the spectral expansion (theorem \ref{thm:spectral_resolvent}) and its consequences (corollary \ref{thm:spectral_expansion}) are proved in section \ref{sec:expansion}.

\textit{3.} The spectral expansions provide linear maps between $\cH$ and a certain space of square integrable functions on which the ``transformed $\oL$'' acts multiplicatively. The fact that these maps are continuous (theorem \ref{thm:TnSn_bounded}), and their restriction certain subspaces of $\cH$ are injective and have a continuous inverse (theorem \ref{thm:UVrestricted}) allow to show that $B_{\oL}$ is bounded and can be extended to a bounded homomorphism $E_{\oL}$ between the Boolean algebra of Borel sets of $\R$ and a Boolean algebra of projections contained in $\cL(\cH)$. Theorems \ref{thm:spectral_measure} and \ref{thm:spectral_operator} follow from these results.

\section{Spectrum of $\oL$ and $\oLa$ \label{sec:spectrum}} 

The spectral properties of $\oL$ and $\oLa$ are related to the properties of the selfadjoint operator
$\Lb$, defined by
\be
\begin{gathered}
\dom(\Lb)=\ran\big(\oLt^{-1/2}(\oLo+P_0)^{-1}\oLt^{-1/2}\big), \\[4pt]
 \Lb f=\oLth\oLo\oLth f, \quad f\in\dom(\Lb),
\end{gathered}
\ee
where $\oLth$ is the positive square root of $\oLt$. It is clear that $\dom(\Lb)$ is dense in $\cH$. That $\Lb$ is selfadjoint is proved by elementary means.

The operator $\oLth$ is relatively bounded with respect to $\Lb$, since for any $f\in\dom(\Lb)$,
\be
\oLth f = (\oLo+P_0)^{-1}\ioLth\Lb f+ P_0\oLth f,
\ee
and $P_0\oLth$ is a bounded finite rank operator. It can be seen similarly that $\oLo$ and $\oLt$ are both relatively bounded with respect to both $\oL$ and $\oLa$.

\begin{thm}
\label{thm:resolvents}
$\rho(\oL)=\rho(\oLa)=\rho(\Lb)$, and if $z\in\rho(\oL)$ then
\be
R_{\oL}(z)=\oLth R_{\Lb}(z)\ioLth, \quad R_{\Lb}(z)=\oLth R_{\oLa}(z)\ioLth.
\label{eq:resolvents}
\ee
\end{thm}
\begin{proof}
We prove first that $\rho(\oL)=\rho(\oLa)$. If $z\in\rho(\oLa)$ then the operator  $\oL-z\id=\oLt(\oLa-z\id)\oLt^{-1}$ is injective and has a dense range, and its inverse, $\oLt R_{\oLa}(z)\oLt^{-1}$, is bounded, since $\oLt$ is relatively bounded with respect to $\oLa$. Hence, $z\in\rho(\oL)$, what implies $\rho(\oLa)\subseteq\rho(\oL)$. 

Take now $z\in\rho(\oL)\setminus\{0\}$. Consider the equation $(\oLa-z\id)g=f$, with $f\in\cH$. Acting with $P_0$ we get $P_0g=-z^{-1}P_0f$. Using $g=P_Rg+P_0g$ and acting with $P_R$ we have
\be
(\oLa-z\id)P_Rg = P_Rf+\frac{1}{z}\oLa P_0f.
\ee
The above equation is a relation between elements of $\ran(P_R)$, and then we can apply $D_{1R}^{-1}$ to both sides, obtaining
\be
(\oL-z\id)D_{1R}^{-1}P_Rg = D_{1R}^{-1}\Big(P_Rf+\frac{1}{z}\oLa P_0f\Big).
\ee
Hence we have
\be
P_Rg = D_1R_{\oL}(z)D_{1R}^{-1}\Big(P_Rf+\frac{1}{z}\oLa P_0f\Big).
\ee
Then, equation $(\oLa-z\id)g=f$ has a unique solution $g\in\cH$ for each $f\in\cH$, which is clearly bounded by $\|g\|\leq c\|f\|$, where $c$ does not depend on $f$. Therefore, $z\in\rho(\oLa)$ and thus $\rho(\oL)\setminus\{0\}\subseteq\rho(\oLa)$. If $0\in\rho(\oL)$ then $\iker(\oLo)=\{0\}$ and $D_1$ is an injective operator with a bounded inverse. Thus it is clear that $\oLa$ is injective and has a bounded inverse and therefore $0\in\rho(\oLa)$. All together we have proved $\rho(\oL)\subseteq\rho(\oLa)$. Thus, $\rho(\oL)=\rho(\oLa)$. 

Now we prove that $\rho(\Lb)\subseteq\rho(\oL)$ and $\rho(\oLa)\subseteq\rho(\Lb)$. For $z\in\rho(\Lb)$ it is clear that $\oL-z\id=\oLth(\Lb-z\id)\ioLth$ is an injective operator with a dense range. Its inverse is bounded since $\oLth$ is relatively bounded with respect to $\Lb$. Hence, $z\in\rho(\oL)$ and $R_{\oL}(z)$ is given by the first of equations (\ref{eq:resolvents}). Similarly, for $z\in\rho(\oLa)$ the operator $\Lb-z\id=\oLth(\oLa-z\id)\ioLth$ is injective and has a dense range. Its inverse is bounded since $\oLth$ is relatively bounded with respect to $\oLa$, what implies that $z\in\rho(\Lb)$ and that $R_{\Lb}(z)$ is given by the second of equations (\ref{eq:resolvents}). 
\end{proof}

We notice that equations (\ref{eq:resolvents}) can be cast to the following useful forms:
\begin{gather}
R_{\oL}(z)\!=\!(\oLo+P_0)^{-1}\ioLth\Big(\id\!+\!zR_{\Lb}(z)\Big)\ioLth\!+\!P_0\oLth R_{\Lb}(z)\ioLth\!,
\label{eq:resolvent_L_bounded} 
\\[2pt]
R_{\Lb}(z)\!=\!\ioLth(\oLo+P_0)^{-1}\Big(\!\id\!+\!zR_{\oLa}(z)\!\Big)\ioLth\!\!+\! P_0\oLth R_{\oLa}(z)\ioLth\!.
\label{eq:resolvent_Lb_bounded}
\end{gather}

Theorem \ref{thm:resolvents} implies that $\sigma(\oL)=\sigma(\oLa)=\sigma(\Lb)\subset\R$.
In the next theorem we analyze the components of the spectrum.

\begin{notat}
\textnormal{
It is convenient to introduce the notation $\ha=(h_1+h_2)/2$, $\hp=h_1h_2$, $h_m=-(h_2-h_1)^2/4$, and the open interval $I=(\hp,\infty)$. We denote by $\sigma_p^{\,\prime}(\oL)$ the set of accumulation points of  $\sigma_p(\oL)$ .
}
\end{notat}

\begin{thm}
\label{thm:spectrumL}
The spectra of $\oL$, $\oLa$, and $\Lb$ are equal, component by component: 
$\sigma_p(\oL)=\sigma_p(\oLa)=\sigma_p(\Lb)$,     
$\sigma_r(\oL)=\sigma_r(\oLa)=\sigma_r(\Lb)=\emptyset$, 
$\sigma_c(\oL)=\sigma_c(\oLa)=\sigma_c(\Lb)=\overline{I}$, 
and
$
\sigma_{ek}(\oL)=\sigma_{ek}(\oLa)=\sigma_{e}(\Lb)=\overline{I}\cup\sigma_p^{\,\prime}(\oL)$, for $1\leq k\leq 5$.
The point spectrum is a countable bounded nowhere dense set which has at most two accumulation points: $\sigma_p^{\,\prime}(\oL)\subseteq\{\hm,\hp\}$.
\end{thm}
\begin{proof}
To prove that the point spectra of the three operators coincide we notice that if $f$ is an eigenfunction of $\Lb$ corresponding to the eigenvalue $\lambda$, then $\oLth f$ and $\ioLth f$ are eigenfunctions of $\oL$ and $\oLa$, respectively, corresponding to the eigenvalue $\lambda$; if $f$ is an eigenfunction of $\oL$ corresponding to the eigenvalue $\lambda$, then $\ioLth f$ and $\oLt^{-1} f$ are eigenfunctions of $\Lb$ and $\oLa$, respectively, corresponding to the eigenvalue $\lambda$; and if $f$ is an eigenfunction of $\oLa$ corresponding to the eigenvalue $\lambda$, then $\oLth f$ and $\oLt f$ are eigenfunctions of $\Lb$ and $\oL$, respectively, corresponding to the eigenvalue $\lambda$. Since $\Lb$ is selfadjoint and $\cH$ separable, $\sigma_p(\Lb)$ is countable. That $\sigma_p(\Lb)$ is a nowhere dense set follows from the fact that each eigenvalue of $\oL$ is a zero of the function $W$ on $\C$ defined in section \ref{sec:green}, which is analytic in the upper half-plane of $\C$ and continuous on the real axis. The zeros of $W$ form a bounded subset of the real axis which may have accumulation points only in $\{\hm,\hp\}$ (see section \ref{sec:green}).

It is now obvious that the residual spectra of the three operators are empty and that their continuous spectra coincide. Since the spectrum of $\oL$ lies on the real axis, the five sets $\sigma_{ek}(\oL)$, $1\leq k\leq 5$, are equal, and are obtained by removing from $\sigma(\oL)$ the isolated eigenvalues, which have finite multiplicity (see point (i) of theorem 1.6 of chapter 9 of Edmunds and Evans \cite{Edmunds2018}, which clearly holds not only for selfadjoint operators in Hilbert spaces but for any operator in a Banach space whose spectrum contains only accumulation points of  its resolvent set).
Evidently, these statements hold also for $\oLa$. Hence, the ten sets $\sigma_{ek}(\oL)$, $\sigma_{ek}(\oLa)$, $1\leq k\leq 5$, are all equal and coincide with $\sigma_e(\Lambda)=\overline{I}\cup\sigma_p^\prime(\oL)$.
\end{proof}

\section{Spectral projections for bounded Borel sets \label{sec:projections_bounded}}

Equation (\ref{eq:resolvents}) for the resolvent of $\oL$ suggests a way to define a family of spectral projections for $\oL$, as follows. Let $\cB_b$ the algebra of bounded Borel sets of $\R$, $\cB$ the Borel $\sigma$-algebra of $\R$, and $\cL(\cH)$ the Banach algebra of bounded linear operators in $\cH$. 
Let $E_{\Lb}:\cB\to\cL(\cH)$ be the resolution of the identity for $\Lb$. Since $\Lb E_{\Lb}(b)$ is a bounded operator if $b\in\cB_b$, we can define a map $B_{\oL}:\cB_b\to\cL(\cH)$ by
\be
B_{\oL}(b) = \oLth E_ {\Lb}(b)\ioLth = \oLth R_{\Lb}(z)(\Lb-z\id)E_ {\Lb}(b)\ioLth
\label{eq:Bdef}
\ee
for $b\in\cB_b$ and $z\in\rho(\oL)$.
The second equality of the above expression shows that  $B_{\oL}(b)$ is a bounded projection for each bounded Borel set  $b$. It is clear that $B_{\oL}$ inherits from $E_{\Lb}$ all the properties of spectral projections, except the uniform boundedness. That is, $B_{\oL}$ satisfies
\begin{gather}
B_{\oL}(b\cap d) = B_{\oL}(b)B_{\oL}(d),
\\[4pt]
B_{\oL}(b\cup d)=B_{\oL}(b)+B_{\oL}(d)-B_{\oL}(b\cap d),
\end{gather}
for $b,d\in\cB_b$. Therefore, the image of $B_{\oL}$ is a Boolean algebra of projections and $B_{\oL}$ is a homomorphism from the boolean algebra $\cB_b$ onto its image. Moreover, if $\{b_i,i\in\N\}$ is a family of pairwise disjoint bounded Borel sets whose union is a bounded set, then
\begin{gather}
B_{\oL}\big(\cup_{i=1}^\infty b_i\big)f=\sum_{i=1}^{\infty} B_{\oL}(b_i)f, \quad f\in \cH.
\end{gather}
This relation implies the continuity of the projections from above and below, that is, if $\{b_i$, $i\in\N\}$ is a descending sequence of bounded Borel sets, then
\be
B_{\oL}\big(\cap_ib_i\big)f = \lim_{i\to\infty}B_{\oL}(b_i)f, \quad f\in\cH,
\ee
and if $\{b_i, i\in\N\}$ is an ascending sequence of bounded Borel sets whose union is a bounded set, then
\be
B_{\oL}\big(\cup_ib_i\big)f = \lim_{i\to\infty}B_{\oL}(b_i)f, \quad f\in\cH.
\ee
An analogue of the Stone formula is also inherited from $E_{\oL}$:
if $a,b\in\R$, with $a<b$, then
\be
\slim_{\epsilon\to0^+}\frac{1}{2\pi\iu}\int_a^b\!\!\Big(R_{\oL}(\lambda+\iu\epsilon)-R_{\oL}(\lambda-\iu\epsilon)\Big)\,d\lambda = \frac{1}{2}\Big(B_{\oL}\big([a,b]\big) + B_{\oL}\big((a,b)\big)\Big).
\label{eq:stone}
\ee

It is straightforward to see that the projections reduce $\oL$, that is, for any $b\in\cB_b$, 
\be
B_{\oL}(b)R_{\oL}(z) f=R_{\oL}(z) B_{\oL}(b)f, \quad f\in\cH, \label{eq:BL_R}
\ee
and if $f\in\dom(\oL)$ then $B_{\oL}(b)f\in\dom(\oL)$ and
\be
B_{\oL}(b)\oL f= \oL B_{\oL}(b)f. \label{eq:BL_L}  
\ee
We also have the following useful formula, analogous to equation (\ref{eq:resolvent_L_bounded}):
\be
B_{\oL}(b) = (\oLo+P_0)^{-1}\ioLth \Lb E_{\Lb}(b)\ioLth + P_0\oLth E_{\Lb}(b)\ioLth.
\label{eq:projection_L_bounded}
\ee
The above equation implies that $\ran\big(B_{\oL}(b)\big)\subset\dom(\oL)$, since $\ran\big(\Lb E_{\Lb}(b)\big)\subset\dom(\Lb)\subset\dom(\oLth )$ and $\ran(P_0)\subset\dom(\oL)$.

We notice the following property, which will be used later in the proof of some results: if $b$ is a non-empty bounded subset of $\R$, then the family of projections associated to the Borel subsets of $b$ is uniformly bounded, in the sense that there is a constant $c$, which may depend on $b$, such that $\|B_{\oL}(e)\|\leq c$ if $e$ is a Borel set contained in $b$. This follows from equation (\ref{eq:projection_L_bounded}), since $\|\Lb E_{\Lb}(e)\|\leq \sup\{|z|, z\in b\}$ for any Borel set $e\subseteq b$.

It is not apparent that $\cB_{\oL}$ is a bounded homomorphism, in the sense that there is $c>0$ such that $\|B_{\oL}(b)\|\leq c$ for all $b\in\cB_b$. In what follows we shall show that $B_{\oL}$ is actually bounded and that it can be extended to a bounded homomorphism of $\cB$ into a Boolean algebra of projections contained in $\cL(\cH)$. To do this we take advantage of the fact that $\oL$ and $\oLa$ are ordinary differential operators and their resolvents are integral operators whose kernel, the Green function, can be studied by standard analytic means. 

\section{The Green function of $\oL$ \label{sec:green}}

The Green function of $\oL$ is constructed from the solutions of the differential equations $\oL u=zu$, with $z\in\rho(\oL)$. Then, the starting point is the analysis of the solutions of these equations.
For $z\in\C$, the fourth order differential equation $\oL u=z u$ is equivalent to the system of four first order equations 
\be
y^\prime=(A+B)y, \label{eq:system}
\ee 
where $y:\R\to\C^4$ is a vector function, $A$ is a $4\times 4$ constant matrix, and $B:\R\to M_4(\C)$ is a matrix-valued function. The only non-zero matrix elements of $A$ and $B$ are
\be
\begin{gathered}
A_{12}=A_{23}=A_{34}=1, \quad A_{41}=z-\hp, \quad A_{43}=-2\ha, \\[6pt]
B_{41}=\pot_2^{\prime\prime}-(h_1+\pot_1)(h_2+\pot_2), 
\;\; B_{42}=2\pot_2^\prime, \;\; B_{43}=h_1+h_2+\pot_1+\pot_2.
\end{gathered}
\label{eq:AB}
\ee
The characteristic polynomial of $A$ is $\pchar(-\iu\mu)-z$, where 
\be
\pchar(\mu)=\mu^4+2\ha\,\mu^2+\hp, \label{eq:pchar}
\ee
and therefore the eigenvalues of $A$ are $\iu\mu_k$, $1\leq k\leq 4$, where $\mu_k$ are the roots of the polynomial $\pchar(\mu)-z$. The four roots are simple if $z\notin\{\hm,\hp\}$. An eigenvector of $A$ corresponding to the eigenvalue $\iu\mu_k$ is the column vector $p_k$ that has $(\iu\mu_k)^{j-1}$ in the $j$-th row ($1\leq j\leq 4$). Let $\Pi(z)$ be the $4\times 4$ matrix whose $k$-th column is $p_k$, so that $\Pi_{jk}(z)=(\iu\mu_k)^{j-1}$. If all the eigenvalues of $A$ are different, then $\Pi(z)$ is regular. Let us also introduce the $4\times 4$ diagonal matrix $D(x,z)$ whose matrix elements are given by $D_{jk}(x,z)=\exp(\iu\mu_kx)\delta_{jk}$ for $1\leq j,k\leq 4$. Bear in mind that the roots $\mu_k$ are functions of $z$.  

Let $e_i$, $1\leq i\leq n$, be the vectors of the canonical basis of $\C^n$. We use the norms for vectors of $v\in\C^n$ and for matrices $M\in M_n(\C)$ given by $|v|=\sum_i|e_i\cdot v|$ and $|M|=\sum_{ij}|M_{ij}|$.
Let us denote by $\C^+$ and $\C^-$, respectively, the open upper and lower half-planes of $\C$.
The following theorem provides us with two fundamental matrices of the system (\ref{eq:system}) for the case $z\in\C\setminus\R$.
\begin{thm}
\label{thm:sol_system}
There are two matrix-valued functions $\Phi_+$ and $\Phi_-$, with domain $\R\times (\C\setminus\R)$ and taking values on $M_4(\C)$, written as
\be
\begin{gathered}
\Phi_+(x,z) = \Pi(z) \big(\id+\Theta_+(x,z)\big) D(x,z), \\[4pt]
\Phi_-(x,z) = \Pi(z) \big(\id+\Theta_-(x,z)\big) D(x,z),
\end{gathered}
\label{eq:fund_mat}
\ee
for $x\in\R$ and $z\in\C\setminus\R$, which have the following properties:
\begin{enumerate}
\item For fixed $z$, $\Phi_+(\cdot,z)$ and $\Phi_-(\cdot,z)$ are  fundamental matrices of the system (\ref{eq:system}).
\item For fixed $x$, $\Phi_+(x,\cdot)$ and $\Phi_-(x,\cdot)$ are analytic.

\item For fixed $z$, $|\Theta_+(x,z)|$  and $|\Theta_-(-x,z)|$ are $o(1)$ as $x\to\infty$.
\item  For fixed $x\geq0$, $|\Theta_+(x,z)|$ and $|\Theta_-(-x,z)|$ are $O(|z|^{-1/4})$ for  $|z|\to\infty$. 
\item The restrictions of $\Phi_+$ and $\Phi_-$ to $\R\times\C^\sigma$, with $\sigma=\pm$, can be extended to continuous functions, $\Phi_+^{(\sigma)}$ and $\Phi_-^{(\sigma)}$, on $\R\times\overline{\C^\sigma}$. 

\item For fixed $\lambda\in\R$, the functions $\Phi_+^{(\sigma)}(\cdot,\lambda)$ and $\Phi_-^{(\sigma)}(\cdot,\lambda)$ are solution matrices of the system (\ref{eq:system}). 

\item For fixed $\lambda\in\R$, $|\Theta_+^{(\sigma)}(x,\lambda)|$ and $|\Theta_-^{(\sigma)}(-x,\lambda)|$ are $o(1)$ as $x\to\infty$.

\item For fixed $x\geq0$ and if $\lambda\in\R$, $|\Theta_+^{(\sigma)}(x,\lambda)|$ and $|\Theta_-^{(\sigma)}(-x,\lambda)|$ are $O(\lambda^{-1/4})$ as $\lambda\to\infty$.
\end{enumerate}
\end{thm}
\begin{proof}
The proof of this theorem is not given here. It only uses standard arguments, similar to those of theorem \ref{thm:sols_system_no_branch}, and it is a modification of the proof of theorem 8.1 of chapter 6 of Coddington and Levinson \cite{Coddington1972}.
\end{proof}

Of special interest are the solutions of $\oL u=\lambda u$ for $\lambda\in [\hp,\infty)$. In this case  $\pchar(\mu)-\lambda$ has four roots, two real and two purely imaginary, which are denoted by $\mu_1 = \iu\mur$, $\mu_2 = \nu$, $\mu_3=-\nu$, $\mu_4=-\iu\mur$,
where $\theta>0$ and $\nu\geq0$ are the functions of $\lambda\in[\hp,\infty)$ given by 
\be
\theta(\lambda) =\left(\sqrt{\lambda+h_\mathrm{a}^2-\hp}+\ha\right)^{1/2}, \quad \nu(\lambda) = \left(\sqrt{\lambda+h_\mathrm{a}^2-\hp}-\ha\right)^{1/2}. \label{eq:thetanu}
\ee
Notice that $\theta$ and $\nu$ are nonnegative increasing functions, with $\nu/\theta\leq 1$ and $\theta\geq\theta_0=\sqrt{2\ha}>0$.

We need fundamental sets of solutions of $\oL u=\lambda u$ that satisfy the estimates (\ref{eq:bounds_rs}) of theorem 
\ref{thm:sols_edo}. To this end, we separate the interval $I$ into two subintervals, $(\hp,\lambda_s]$ and $(\lambda_s,\infty)$, where $\lambda_s>\hp$ is sufficiently large.

\begin{thm}
\label{thm:sols_edo}
Let $\lambda_s>\hp$ sufficiently large. For $1\leq k\leq 4$ there are functions $\phip_k:\R\times \bar{I}\to\C$ and $\phim_k:\R\times \bar{I}\to\C$ such that:
\begin{enumerate}
\item
For fixed $\lambda\in I$, $\{\phip_k(\cdot,\lambda), 1\leq k\leq 4\}$ and $\{\phim_k(\cdot,\lambda), 1\leq k\leq 4\}$ are fundamental sets of solutions of $\oL u = \lambda u$. 
\item For fixed $x\in\R$, $\phip_k(x,\cdot)$ and $\phim_k(x,\cdot)$ are analytic in $I\setminus\{\lambda_s\}$.
\item The restrictions to $\R\times[\hp,\lambda_s]$ and to $\R\times(\lambda_s,\infty)$ of each of these functions are continuous functions.
\item For each $x\in\R$, the limits of $\phip_k(x,\lambda)$ and $\phim_k(x,\lambda)$ as $\lambda\to\lambda_s^+$ exist and provide  fundamental sets of solutions (as in point 1) of $\oL u = \lambda_s u$.
\item
For $x\in\R$ and $\lambda\in \bar{I}$ the functions have the form
\be
\begin{gathered}
\phip_k(x,\lambda) = \exp(\iu\mu_kx)\big(1+r_k(x,\lambda)\big), \\[4pt] \phim_k(x,\lambda) = \exp(\iu\mu_kx)\big(1+s_k(x,\lambda)\big), 
\end{gathered}
\label{eq:sols_edo}
\ee
where $r_k$ and $s_k$ satisfy the bounds
\be
\begin{gathered}
|r_k(x,\lambda)|  \leq c\, \lambda^{-1/4} (1+x)^{-1}, \quad  x\geq 0, \\[4pt]
|s_k(x,\lambda)|  \leq c\, \lambda^{-1/4} (1-x)^{-1}, \quad  x\leq 0,
\end{gathered}
\label{eq:bounds_rs}
\ee
where $c>0$ is a constant.
\end{enumerate}
\end{thm}
\begin{proof}
The proof of this theorem relies on standard arguments similar to those of theorem 8.1 of chapter 3 of Coddington and Levinson \cite{Coddington1972}. Since this theorem is a key point for this work, we provide the proof in \ref{app:sols_edo}.
\end{proof}

In the remaining of the paper the functions $\phip_k$ and $\phim_k$, $1\leq k\leq 4$, have the properties of theorem \ref{thm:sols_edo}.
There is a function $C:\bar{I}\to GL(4,\C)$, bounded, analytic in $I\setminus\{\lambda_s\}$, such that  
\be
\phim_k(\cdot,\lambda) = \sum_{j=1}^4 c_{jk}(\lambda)\,\phip_j(\cdot,\lambda),\;\;1\leq k\leq 4, \quad \lambda\in \bar{I}, \label{eq:part:rel_phip_phim}
\ee
where $c_{jk}(\lambda)$ are the matrix elements $C(\lambda)$. For $\lambda\to\infty$ we have the estimate $C(\lambda) =\id+ O(1/\lambda^{1/4})$, which follows from the estimates of theorem \ref{thm:sols_edo}.

Since the differential expression $\oLa u$ has the same structure as $\oL u$, there is an analogue to theorem \ref{thm:sols_edo} concerning the solutions of $\oLa u=\lambda u$. That is, there are functions $\phip_k^*$ and $\phim_k^*$, $1\leq k\leq 4$, with domain $\R\times \bar{I}$, such that, for each $\lambda\in \bar{I}$,  $\{\phip^*_k(\cdot,\lambda), 1\leq k\leq 4\}$ and $\{\phim^*_k(\cdot,\lambda), 1\leq k\leq 4\}$ are fundamental sets of solutions of $\oLa u = \lambda u$. The functions $\phip^*_k$ and $\phim^*_k$ have the properties of $\phip_k$ and $\phim_k$, respectively, listed in points 2-5 of theorem \ref{thm:sols_edo}. The functions corresponding to $r_k$ and $s_k$ of theorem \ref{thm:sols_edo}, and to $c_{jk}$ of equation (\ref{eq:part:rel_phip_phim}), are represented by $r^*_k$, $s^*_k$, and $c_{jk}^*$, respectively.

\begin{lem}
\label{thm:sols_bounded_dim}
If $\lambda\in I$ is not an eigenvalue of $\oL$, the linear space of bounded solutions of $\oL u = \lambda u$ has dimension two, and the linear space of bounded solutions of $\oLa u = \lambda u$ has also dimension two.
\end{lem}
\begin{proof}
In this proof the indices $l$ and $k$ take the values $1\leq l\leq 3$ and $2\leq k\leq 4$. Any bounded solution of $\oL u = \lambda u$ is of the form 
$\phi=\sum_l\alpha_l\phip_l(\cdot,\lambda)=\sum_k\beta_k\phim_k(\cdot,\lambda)$, where $\alpha_l$ and $\beta_k$ are complex numbers, which are not necessarily independent, since relation (\ref{eq:part:rel_phip_phim}) implies $\alpha_l = \sum_k c_{lk}\beta_k$ and $\sum_k c_{4k} \beta_k=0$. Unless $c_{42}=c_{43}=c_{44}=0$, the last constraint on the values of $\beta_k$ determines a two dimensional linear subspace of the linear space of all triples $(\beta_2,\beta_3,\beta_4)$, and thus the linear space of bounded solutions of $\oL u = \lambda u$ has dimension two.
In the case $c_{42}=c_{43}=c_{44}=0$ the values of $\beta_k$ are unconstrained and  the linear space of bounded solutions of $\oL u = \lambda u$ has dimension three. But then $\phim_4 = c_{41}\phip_1$, with $c_{41}\neq 0$, and then $\phim_4$ is a nontrivial square integrable solution of $\oL u=\lambda u$, what implies that $\lambda$ is an eigenvalue of $\oL$, a possibility excluded by the hypothesis. The proof of the statement for the solutions of $\oLa u = \lambda u$ is similar. 
\end{proof}

The Green matrix of the system (\ref{eq:system}) is built as follows. For $z\in\C\setminus\R$ we order the roots of $\pchar(\mu)-z$ so that $\Imag\,\mu_1\geq\Imag\,\mu_2 > 0 > \Imag\,\mu_3\geq\Imag\,\mu_4$. This ordering determines the ordering of the columns of the fundamental matrices $\Phi_+(\cdot,z)$ and $\Phi_-(\cdot,z)$. For $x\in\R$ and $z\in\C\setminus\R$ we define the matrix
\be
\Phi(x,z)= \Phi_+(x,z)P_+ + \Phi_-(x,z)P_-,
\ee
where $P_+=\mathrm{diag}(1,1,0,0)$ and $P_-=\id-P_+$. The matrix $\Phi(x,z)$ is regular for all $x\in\R$, and its determinant, $W(z)$, is independent of $x$. For $x,\tau\in\R$ and $z\in\C\setminus\R$ the Green matrix of the system (\ref{eq:system}) is
\be
K(x,\tau,z) = \left\{
\begin{array}{rl}
\Phi_+(x,z)P_+\Phi^{-1}(\tau,z), & x\geq\tau, \\[4pt]
-\Phi_-(x,z)P_-\Phi^{-1}(\tau,z), & x<\tau,
\end{array}
\right.
\ee
and satisfies $|K(x,\tau,z)|\leq k\exp(-\delta|x-\tau|)$, for $x,\tau\in\R$, where $k>0$ and $\delta>0$ depend on $z$ but not on $x$ or $\tau$ \cite{Kemp1960}. 
The Green function of $\oL$ is obtained from the Green matrix of the system (\ref{eq:system}) as $G(x,\tau,z)=K_{14}(x,\tau,z)$.

The properties of $\Phi_+$ and $\Phi_-$ (theorem \ref{thm:sol_system}) imply that the function $W$, defined in $\C\setminus\R$, is analytic and can be continuously extended to the real axis both from the upper and lower open half-planes of $\C$. These continuous extensions define the functions $W_+$ and $W_-$ on $\R$ by
\be
W_{\pm}(\lambda)=\lim_{\epsilon\to0^+} W(\lambda\pm\iu\epsilon), \quad \lambda\in\R.
 \ee
The properties of $\Phi_+$ and $\Phi_-$ also guarantee that for $\lambda\in\R$ the limits
\be
G_+(x,\tau,\lambda)=\lim_{\epsilon\to0^+}G(x,\tau,\lambda+\iu\epsilon), \quad G_-(x,\tau,\lambda)=\lim_{\epsilon\to0^+}G(x,\tau,\lambda-\iu\epsilon),
\ee
exist if $W_+(\lambda)\neq 0$ and $W_-(\lambda)\neq 0$, respectively. 

Let $N_+$ and $N_-$ be the set of zeros of $W_+$ and $W_-$, respectively, and $N=N_+\cup N_-$. Since $W(z)=1+o(1)$ as $|z|\to\infty$, $N_+$, $N_-$, and $N$ are bounded closed nowhere dense subsets of $\R$, and are contained in the spectrum of $\oL$. For $n\in\N$, let us define the sets 
\be
M_n=\big\{\lambda\in I \st |W_+(\lambda)|>1/(2n), |W_-(\lambda)|>1/(2n)\big\},
\ee
and $N_n=\sigma(\oL)\setminus M_n$, so that $\sigma(\oL)=M_n\cup N_n$. Notice that $M_n$ is open and unbounded and $N_n$ is closed and bounded, and that $\overline{M_n}\subseteq M_{n+1}$ and $N_{n+1}\subseteq N_n$. Clearly $\cap_n N_n=N$. We define $M=\cup_n M_n$, and then we have $M\cap N=\emptyset$ and $M\cup N= \sigma(\oL)$. We take $\lambda_s$ (theorem \ref{thm:sols_edo}) large enough so that $[\lambda_s,\infty)\subseteq M_n$ for all $n\in\N$.

The expression $G_+(x,\tau,\lambda)-G_-(x,\tau,\lambda)$, for $x,\tau\in\R$ and $\lambda\in M$, defines a continuous function on $\R\times\R\times M$, bounded on each subset $\R\times\R\times M_n$, with $n\in\N$. For fixed $\tau$ and $\lambda$ the function $G_+(\cdot,\tau,\lambda)-G_-(\cdot,\tau,\lambda)$ is a bounded solution of $\oL u=\lambda u$, and for fixed $x$ and $\lambda$ the function $\overline{G_+(x,\cdot,\lambda)}-\overline{G_-(x,\cdot,\lambda)}$ is a bounded solution of $\oLa u=\lambda u$. Therefore, there are continuous functions $\ef_j:\R\times M\to\C$ and $\aef_j:\R\times M\to\C$, ($j=1,2$), bounded on each subset $\R\times M_n$ ($n\in\N$), such that, for each fixed $\lambda\in M$, $\ef_1(\cdot,\lambda)$ and $\ef_2(\cdot,\lambda)$ are two linearly independent bounded solutions of $\oL u=\lambda u$, and $\aef_1(\cdot,\lambda)$ and $\aef_2(\cdot,\lambda)$ are two bounded solutions of $\oLa u=\lambda u$, and
\be
G_+(x,\tau,\lambda)-G_-(x,\tau,\lambda) = 2\pi\iu \sum_{j=1}^2 \ef_j(x,\lambda)\,\overline{\aef_j(\tau,\lambda)}. \label{eq:phidef}
\ee
The asymptotic analysis shows \cite{Kemp1960} that for fixed $x$ and $\tau$ and $\lambda\to\infty$
\be
G_+(x,\tau,\lambda)-G_-(x,\tau,\lambda) = \frac{2\iu\cos\big(\nu(\lambda)(x-\tau)\big)}{\pchar^\prime\big(\nu(\lambda)\big)}\Big(1+o(1)\Big), \label{eq:asymp_phi}
\ee
where $\nu(\lambda)$ is given by the second of equations (\ref{eq:thetanu}).  Then, we can choose the functions $\ef_j$ so that for $n\in\N$ 
\be
|\ef_j(x,\lambda)|\leq \frac{c_n}{\lambda^{3/8}}, \quad |\aef_j(x,\lambda)|\leq \frac{c_n}{\lambda^{3/8}}, \quad x\in\R, \quad \lambda\in M_n, \label{eq:bounds_phi}
\ee
where the sequence $c_n$ may be unbounded.

\begin{lem}
\label{thm:sols_bounded}
For $1\leq j\leq 2$, $1\leq l\leq 3$, and $2\leq k\leq 4$, there are functions $\alpha_{jl}$, $\beta_{jk}$, $\alpha^*_{jl}$, $\beta^*_{jk}$, with domain $M$ , continuous on $M\setminus\{\lambda_s\}$, and bounded on each subset $\overline{M_n}$, with $n\in\N$, such that for each $x\in\R$ and $\lambda\in M$,
\begin{gather}
\ef_j\big(x,\lambda\big) = \sum_{l=1}^3 \alpha_{jl}(\lambda) \phip_l\big(x,\lambda\big) = \sum_{k=2}^4 \beta_{jk}(\lambda) \phim_k\big(x,\lambda\big),
 \label{eq:sols_bounded_L} \\[4pt]
 \aef_j\big(x,\lambda\big) = \sum_{l=1}^3 \alpha^*_{jl}(\lambda) \phip^*_l\big(x,\lambda\big) = \sum_{k=2}^4 \beta^*_{jk}(\lambda) \phim^*_k\big(x,\lambda\big).
 \label{eq:sols_bounded_La}
\end{gather}
Moreover, there is $c>0$ such that, for each $\lambda \geq \lambda_s$, 
\be
 \lambda^{p_l}|\alpha_{jl}(\lambda)|\leq c, \;\;  \lambda^{p_k}|\beta_{jk}(\lambda)|\leq c,  \;\; \lambda^{p_l}|\alpha^*_{jl}(\lambda)|\leq c, \;\;  \lambda^{p_k}|\beta^*_{jk}(\lambda)|\leq c,  \label{eq:bounds_alpha_beta}
\ee
where $p_i=5/8$ for $i=1,4$ and $p_i=3/8$ for $i=2,3$.
\end{lem}
\begin{proof}
In this proof the indices take the values $1\leq j\leq 2$, $1\leq l\leq 3$, and $2\leq k\leq 4$.
For $\lambda\in M$, $\ef_j(\cdot,\lambda)$ is a unique linear combination of $\phip_m(\cdot,\lambda)$, with $1\leq m\leq 4$. The bounded\-ness of $\ef_j(\cdot,\lambda)$ for each $\lambda\in M$ implies that the coefficient of $\phi_4(\cdot,\lambda)$ has to be zero. This linear combination defines the functions $\alpha_{jl}$. The functions $\ef_j$, $\phip_l$ are continuous and bounded on each subset $\overline{\R^+}\times(M_n\setminus\{\lambda_s\})$. By lemma \ref{thm:sols_edo} and lemma \ref{thm:continuity_coefficients} of appendix \ref{app:ancillary} the functions $\alpha_{jl}$ are continuous on $M_n\setminus\{\lambda_s\}$, for each $n\in\N$.
Since $\overline{M_n}\subset M_{n+1}$, the restriction of $\ef_j$ to $\R\times \overline{M_n}$ is continuous and bounded. By lemma \ref{thm:sols_edo}, the restrictions of $\phip_l$ to $\overline{\R^+}\times \big(\overline{M_n}\cap [\hp,h_s]\big)$ are continuous and bounded, and the restrictions of $\phip_l$ to $\overline{\R^+}\times \big(\overline{M_n}\cap (h_s,\infty)\big)$ can be extended to continuous bounded functions on $\overline{\R^+}\times \big(\overline{M_n}\cap [h_s,\infty)\big)$. Then, from lemma \ref{thm:continuity_coefficients} we infer that the functions $\alpha_{jl}$ are bounded on each bounded subset of $\overline{M_n}$. To see that they are bounded on $\overline{M_n}$, take $\lambda\in \overline{M_n}$ large enough and consider the sequences of positive numbers $x_m=2\pi m/\nu(\lambda)$, $y_m=x_m+\pi/\big(2\nu(\lambda)\big)$, $m\in\N$. From the boundedness of the sequences $\ef_j(x_m,\lambda)$ and $\ef_j(y_m,\lambda)$ it is easily obtained that the restrictions of $\alpha_{j2}$ and $\alpha_{j3}$ to $\overline{M_n}$ are bounded functions. Then, from the boundedness of $\ef_j(0,\lambda)$ and the estimates (\ref{eq:bounds_rs}), we find that $\alpha_{j1}$ is also bounded on $\overline{M_n}$. 
The proofs of the continuity and boundedness of $\beta_{jk}$, $\alpha_{jl}^*$, and $\beta_{jk}^*$ are similar.

The estimates (\ref{eq:bounds_phi}) imply that that $\alpha_{j2}(\lambda)$, $\alpha_{j3}(\lambda)$, $\beta_{j2}(\lambda)$, and $\beta_{j3}(\lambda)$ are $O\big(1/\lambda^{3/8}\big)$ for $\lambda\to\infty$. From the relations 
$\alpha_{jl}(\lambda)=\sum_{k} c_{lk}(\lambda)\beta_{jk}(\lambda)$ and $\sum_{k} c_{4k}(\lambda)\beta_{jk}(\lambda) = 0$,
 and the asymptotics $c_{im}(\lambda)=\delta_{im}+O(1/\lambda^{1/4})$, $1\leq i,m\leq 4$, we find that $\alpha_{j1}(\lambda)$ and $\beta_{j4}(\lambda)$ are both $O(1/\lambda^{5/8})$. The estimates (\ref{eq:bounds_alpha_beta}) follow taking into account the boundedness of $\alpha_{jl}$ and $\beta_{jk}$ on each set $\overline{M_n}$.
The estimates for $\alpha^*_{jl}$ and $\beta^*_{jk}$ are proved in a similar way.
\end{proof}

\section{Spectral expansion \label{sec:expansion}} 

For $z\in\rho(\oL)$ the resolvent $R_{\oL}(z)$ is an integral operator whose kernel is the Green function, $G(x,\tau,z)$.
We obtain an integral formula for the Green function in the following standard way. Let $\lambda_0$ be a real number such that $(\lambda_0,\infty)$ contains the spectrum of $\oL$ and the point $\hm$, and for each $\epsilon>0$ let $C_\epsilon$ be the curve on the complex plane formed by the points which are at a distance $\epsilon$ from $(\lambda_0,\infty)$, oriented in clockwise sense. The curve separates the complex plane into two disjoint subsets, one of which contains the spectrum of $\oL$. For fixed $x$ and $\tau$, the function $G(x,\tau,\cdot)$ is analytic in the subset that does not contain the spectrum of $\oL$. From the behavior of $G(x,\tau,\xi)$ for large $|\xi|$ in the subset that does not contain the spectrum of $\oL$ and the Cauchy theorem we obtain, for $z\notin(\lambda_0,\infty)$ and $\epsilon$ sufficiently small,
\be
G(x,\tau,z) = \frac{1}{2\pi\iu} \int_{C_\epsilon}\frac{G(x,\tau,\xi)}{\xi-z}\,d\xi.
\ee
We divide $C_\epsilon$ into three pieces, as follows. Take $\eta>0$. The first piece, $S_{\epsilon}$, is the subset of points $\xi\in C_\epsilon$ such that $\Real\,\xi<\lambda_0$. The second piece consists of the points $\xi\in C_\epsilon$ with $\lambda_0\leq\Real\,\xi\leq\hp+\eta$, and the third piece is formed by the points of $C_\epsilon$ with real part larger than $\hp+\eta$. Then, we have 
\be
\begin{gathered}
G(x,\tau,z) = \frac{1}{2\pi\iu}\int_{S_{\epsilon}}\frac{G(x,\tau,\xi)}{\xi-z}d\xi+
\frac{1}{2\pi\iu}\int_{\lambda_0}^{\hp+\eta}F_\epsilon(x,\tau,z,\lambda)d\lambda
\\[4pt]
+ \frac{1}{2\pi\iu}\int_{\hp+\eta}^{\infty}\!\!\!\!F_\epsilon(x,\tau,z,\lambda)d\lambda,
\end{gathered}
\label{eq:resolvent_pieces}
\ee
where
\be
F_\epsilon(x,\tau,z,\lambda)=\frac{G(x,\tau,\lambda+\iu\epsilon)}{\lambda+\iu\epsilon-z}-\frac{G(x,\tau,\lambda-\iu\epsilon)}{\lambda-\iu\epsilon-z}. \label{eq:Fdef}
\ee

The next lemma is used in the proof of part of theorem \ref{thm:spectral_resolvent}.

\begin{lem}
\label{thm:resolvent_integral_ab}
For each $a,b\in\R$, with $a<b$, and $z\in\rho(\oL)\setminus[a,b]$, we have
\be
\begin{gathered}
\slim_{\epsilon\to0^+} \frac{1}{2\pi\iu} \int_{a}^{b}\left(\frac{R_{\oL}(\lambda+\iu\epsilon)}{\lambda+\iu\epsilon-z}-\frac{R_{\oL}(\lambda-\iu\epsilon)}{\lambda-\iu\epsilon-z}\right)
d\lambda \\[4pt]
= \frac{1}{2}R_{\oL}(z)\Big(B_{\oL}\big((a,b)\big)+B_{\oL}\big([a,b]\big)\Big).
\end{gathered}
\label{eq:resolvent_integral_ab}
\ee
\end{lem}
\begin{proof}
Using equation (\ref{eq:resolvent_L_bounded}) we can rewrite the left-hand side of equation (\ref{eq:resolvent_integral_ab}) as the strong limit as $\epsilon\to0^+$ of
\be
\begin{gathered}
 (\oLo+P_0)^{-1}\ioLth  \big(A_\epsilon(z)+zB_\epsilon(z)\big) \ioLth+ P_0 \oLth B_\epsilon(z) \ioLth  + \\[4pt]
h_\epsilon(z) (\oLo+P_0)^{-1}\oLt^{-1},
\end{gathered}
\label{eq:resolvent_integral_ab_cont}
\ee
where
\begin{gather}
A_\epsilon(z)=\frac{1}{2\pi\iu}\int_{a}^{b} \Big(R_{\Lb}(\lambda+\iu\epsilon)-R_{\Lb}(\lambda-\iu\epsilon)\Big)d\lambda, \\[4pt]
B_\epsilon(z) = \frac{1}{2\pi\iu}\int_{a}^{b} \Big(\frac{R_{\Lb}(\lambda+\iu\epsilon)}{\lambda+\iu\epsilon-z}-\frac{R_{\Lb}(\lambda-\iu\epsilon)}{\lambda-\iu\epsilon-z}\Big)\,d\lambda, \\[4pt]
h_\epsilon(z)=\frac{1}{2\pi\iu}\int_{a}^{b}\left(\frac{1}{\lambda+\iu\epsilon-z}-\frac{1}{\lambda-\iu\epsilon-z}\right)d\lambda.
\end{gather}
It is clear that $h_\epsilon(z)$ vanishes as $\epsilon\to0^+$ since $z\notin[a,b]$. The strong limit of $A_\epsilon(z)$ is given by the  analogue of Stone's formula (\ref{eq:stone}). To obtain an expression for $B_\epsilon(z)$ we take advantage of the operational calculus for the selfadjoint operator $\Lb$. On $[a,b]\times\R$ we define the bounded Borel function $f_\epsilon$ by
\be
f_\epsilon(\lambda,\xi) = \frac{1}{2\pi\iu}\left(\frac{1}{(\lambda+\iu\epsilon-z)(\xi-\lambda-\iu\epsilon)}-\frac{1}{(\lambda-\iu\epsilon-z)(\xi-\lambda+\iu\epsilon)}\right),
\ee
for $\lambda\in[a,b]$ and $\xi\in\R$. From $f_\epsilon$ we define the function $g_\epsilon$ on $\R$ by
\be
g_\epsilon(\xi) =\int_a^b f_\epsilon(\lambda,\xi)\,d\lambda, \quad \xi\in\R.
\ee
Then, by lemma 4.1 of Teschl \cite{Teschl2009},  $B_\epsilon(z)=g_\epsilon(\Lb)$. The integral defining $g_\epsilon(\xi)$ can be performed by elementary means, and it is easy to see that $g_\epsilon$ is bounded uniformly in $\epsilon$ if $z\neq a,b$ and $\epsilon>0$ is sufficiently small.  For $z\notin[a,b]$ we obtain
\be
g(\xi) = \lim_{\epsilon\to0^+} g_\epsilon(\xi) =\frac{1}{2}\Big(\chi_{(a,b)}(\xi)+\chi_{[a,b]}(\xi)\Big) \frac{1}{\xi-z}, \quad \xi\in\R,
\ee
where $\chi_G$ is the characteristic function of the set $G$. The operational calculus for $\Lb$ (lemma 4.2 and theorem 3.1 of Teschl \cite{Teschl2009}) gives
\be
\slim_{\epsilon\to0^+} B_\epsilon(z) = \slim_{\epsilon\to0^+} g_\epsilon(\Lb) = g(\Lb) = R_{\Lb}(z)\frac{1}{2}\Big(E_{\Lb}\big((a,b)\big)+E_{\Lb}\big([a,b]\big)\Big).
\ee
Then, the left-hand side of equation (\ref{eq:resolvent_integral_ab_cont}) is equal to
\be
\begin{gathered}
\frac{1}{2} (\oLo+P_0)^{-1}\ioLth  \big(\id+zR_{\Lb}(z)\big) \Big(E_{\Lb}\big((a,b)\big)+E_{\Lb}\big([a,b]\big)\Big) \ioLth + \\[4pt]
+\frac{1}{2} P_0 \oLth R_{\Lb}(z) \Big(E_{\Lb}\big((a,b)\big)+E_{\Lb}\big([a,b]\big)\Big) \ioLth  ,
\end{gathered}
\ee
which, using $\id+zR_{\Lb}(z)=\Lb R_{\Lb}(z)$, the definition of $\Lb$, and $(\oLo+P_0)^{-1}\oLo f=P_Rf$ for $f\in\dom(\oLo)$, can be cast to the form
\be
\oLth R_{\Lb}(z)\frac{1}{2}\Big(E_{\Lb}\big((a,b)\big)+E_{\Lb}\big([a,b]\big)\Big) \ioLth.
\ee
The statement of the lemma is obtained from the above expression by inserting $\ioLth\oLth$ between $R_{\Lb}(z)$ and the spectral projections, what is allowed since the range of $E_{\Lb}([a,b])$ is contained in the domain of $\oLth$.
\end{proof}

For $f\in\cS$ we define on $M$ the the functions $T_jf$, $S_jf$, $j=1,2$, by
\be
T_jf(\lambda)=\int_{-\infty}^\infty f(\tau)\,\overline{\aef_j(\tau,\lambda)}\,d\tau, \quad
S_jf(\lambda)=\int_{-\infty}^\infty f(\tau)\,\overline{\ef_j(\tau,\lambda)}\,d\tau, 
\label{eq:hatf}
\ee
for $\lambda\in M$. For each $n\in\N$ the restrictions of the functions $\ef_j$ and $\aef_j$ to $\R\times M_n$ are continuous and bounded, and therefore the restrictions of $T_jf$ and $S_jf$ ($j=1,2$) to $M_n$ are continuous, bounded, integrable and square integrable functions, since, for $\lambda\in M_n$,
\be
\lambda \big|T_jf(\lambda)\big| = \left|\int_{-\infty}^\infty\oL f(x) \overline{\aef_j(x,\lambda)}\,dx\right| \leq \frac{c_n}{\lambda^{3/8}} \int_{\R} |\oL f(x)|\,dx.
\ee
To derive the above inequality we use of integration by parts, which is allowed since $f\in\cS$, the estimates (\ref{eq:bounds_phi}), and the fact that $\oL f\in L^1(\R)$ if $f\in\cS$. The analogous properties for $S_jf$ are shown similarly.

\begin{rmk}
\textnormal{
The bounds of the restrictions of $\ef_j$ and $\aef_j$ to $\R\times M_n$ need not be uniform in $n$, and thus $T_jf$ and $S_jf$ may be unbounded on $M$ (but they are continuous). Therefore, they do not belong to $L^2(M)$ necessarily. 
}
\end{rmk}

\begin{thm}
\label{thm:spectral_resolvent}
Let $n\in\N$. For $z\in\rho(\oL)\setminus(\lambda_0,\infty)$, for each $f\in\cS$, and for $x\in\R$, 
\begin{gather}
R_{\oL}(z)f(x) = R_{\oL}(z)B_{\oL}(N_n) f(x) 
+\int_{M_n}\sum_{j=1}^2\frac{\ef_j(x,\lambda)T_jf(\lambda)}{\lambda-z}\,d\lambda,
\label{eq:spectral_resolvent_L}
\\[4pt]
R_{\oLa}(z)f(x) = R_{\oLa}(z)B^*_{\oL}(N_n) f(x) 
+\int_{M_n}\sum_{j=1}^2\frac{\aef_j(x,\lambda)S_jf(\lambda)}{\lambda-z}\,d\lambda.
\label{eq:spectral_resolvent_La}
\end{gather}
\end{thm}
\begin{proof}
We prove relation (\ref{eq:spectral_resolvent_L}). The proof of relation (\ref{eq:spectral_resolvent_La}) is similar. Take $f\in\cS$, multiply equation both sides of  (\ref{eq:resolvent_pieces}) by $f(\tau)$ and integrate over $\tau\in\R$. 
Hence we get $R_{\oL}(z)f(x)=A_\epsilon f(x) +B_{\eta\epsilon}f(x)+C_{\eta\epsilon}f(x)$, where each term of the right-hand side of this equation is the contribution to $R_{\oL}(z)f(x)$ of the corresponding term of the right-hand side of equation (\ref{eq:resolvent_pieces}).  It is clear that $A_\epsilon f$ vanishes as $\epsilon\to0^+$, and that
\be
B_{\eta\epsilon}f=\frac{1}{2\pi\iu}\int_{\lambda_0}^{\hp+\eta}\left(\frac{R_{\oL}(\lambda+\iu\epsilon)}{\lambda+\iu\epsilon-z}-\frac{R_{\oL}(\lambda-\iu\epsilon)}{\lambda-\iu\epsilon-z}\right)
d\lambda\,f.
\label{eq:Bne}
\ee
Applying lemma \ref{thm:resolvent_integral_ab}, we get that the limit of $B_{\eta\epsilon}f$ as $\epsilon\to0^+$ and $\eta\to0^+$ is $R_{\oL}(z)B_{\oL}\big([\lambda_0,\hp]\big)f$. 

To obtain $C_{\eta\epsilon}f$ we define $M_{n\eta}=M_n\cap I_\eta$ and $N_{n\eta}=N_n\cap I_\eta$, where $I_\eta=(\hp+\eta,\infty)$. Notice that $M_{n\eta}\subseteq M_{n\eta^\prime}$ and $N_{n\eta}\subseteq N_{n\eta^\prime}$ if $\eta>\eta^\prime$, and $\cup_{\eta>0} M_{n\eta}=M_n$ and $\cup_{\eta>0} N_{n\eta}=N_n\cap I$. The integral over $I_\eta$ that defines $C_{\eta\epsilon}f$ is equal to the sum of an integral over $N_{n\eta}$ and an integral over $M_{n\eta}$, so that we have  $C_{\eta\epsilon}f=D_{\eta\epsilon}f+E_{\eta\epsilon}f$, where $D_{\eta\epsilon}f$ and $E_{\eta\epsilon}f$ are the contribution of the integrals over $N_{n\eta}$ and $M_{n\eta}$, respectively. 

Let us study first $D_{\eta\epsilon}f$, which is given by an expression like the right-hand side of equation (\ref{eq:Bne}), with the integral over $N_{n\eta}$ instead of over $(\lambda_0,\hp+\eta)$. Let us introduce the bounded interval $I_{s\eta}=(\hp+\eta,\lambda_s)$. It is clear that $N_{n\eta} = N_n\cap I_{s\eta}$, and therefore the integral over $N_{n\eta}$ is equal to the integral over $I_{s\eta}$ minus the integral over $I_{s\eta}\cap N_n^c$. Since $N_n^c$ is an open set, it is equal to the union of an at most countable family of disjoint open intervals, $I_k$, with $k$ in some subset $K$ of $\N$. Hence, $I_{s\eta}\cap N_n^c=\cup_k (I _{s\eta}\cap I_k)$. Thus $D_{\eta\epsilon}f$ is expressed as an integral over $I_{s\eta}$ minus the sum of integrals over $I _{s\eta}\cap I_k$. Applying lemma \ref{thm:resolvent_integral_ab} to each of these integrals we get that the limit as $\epsilon\to0^+$ of $D_{\eta\epsilon}f$ is
\be
\frac{1}{2}R_{\oL}(z)\Big(
B_{\oL}(I_{s\eta})+B_{\oL}(\overline{I_{s\eta}})-\sum_{k\in K} \Big(B_{\oL}\big( I _{s\eta}\cap I_k \big)
+B_{\oL}\big(\overline{ I _{s\eta}\cap I_k } \big)\Big)\Big) f. \label{eq:deta}
\ee
It is clear that $\sigma_p(\oL)\cap I$ is contained in the interior of $N_n$, and therefore $\sigma_p(\oL)\cap\overline{I_k}=\emptyset$ for all $k\in K$. This implies $B_{\oL}\big( \overline{I _{s\eta}\cap I_k}\big)=B_{\oL}(\overline{I_{s\eta}}\cap I_k)$, and then from the expression (\ref{eq:deta}) we get
\be
\begin{gathered}
\lim_{\epsilon\to0^+} D_{\eta\epsilon}f = 
\frac{1}{2}R_{\oL}(z)\Big(
B_{\oL}(I_{s\eta}\cap N_n)+B_{\oL}(\overline{I_{s\eta}}\cap N_n)
\Big)f \\
= \frac{1}{2}R_{\oL}(z)\Big(
B_{\oL}(I_{s\eta})+B_{\oL}(\overline{I_{s\eta}})
\Big)B_{\oL}(N_n)f.
\end{gathered}
\ee
The limit $\eta\to0^+$ of the above expression is equal to $R_{\oL}(z)B_{\oL}(I\cap N_n)f$. Since the spectral projections $B_{\oL}(b)$ are supported on the spectrum of $\oL$, and since $\sigma(\oL)\cap[\lambda_0,\hp]\subset N_n\cap [\lambda_0,\hp]$, we obtain that  the limit $\epsilon\to0^+$ and $\eta\to0^+$ of $B_{\eta\epsilon}f+D_{\eta\epsilon}f$ is equal to $R_{\oL}(z)B_{\oL}(N_n)$, which is the first term of the right-hand side of equation (\ref{eq:spectral_resolvent_L}).

It remains to analyze the contribution of $E_{\eta\epsilon}f$, which is given by
\be
E_{\eta\epsilon}f(x) =
\frac{1}{2\pi\iu}\int_{\R}\int_{M_{n\eta}} F_\epsilon(x,\tau,z,\lambda) f(\tau)\,d\lambda\,d\tau.
\ee
By Fubini's theorem, the iterated integrals of the right-hand side are equal to an integral over $M_{n\eta}\times\R$, the integrand of which, $F_{\epsilon}(x,\tau,z,\lambda)f(\tau)$, is bounded by a function $g(\lambda,z)|f(\tau)|$ independent of $\epsilon$ and integrable on $M_{n\eta}\times\R$. Hence, the limit $\epsilon\to0^+$ can be interchanged with the integral and we obtain the integral over $M_{n\eta}\times\R$ of the function
\be
F(\lambda,\tau) = \lim_{\epsilon\to 0^+}F_{\epsilon}(x,\tau,z,\lambda)f(\tau) = \frac{ \sum_{j=1}^2\ef_j(x,\lambda)\overline{\aef_j(\tau,\lambda)}}{\lambda-z}f(\tau),
\ee
for $(\lambda,\tau)\in M_{n\eta}\times\R$. In the second equality of the above expression we used equation (\ref{eq:phidef}). Now we can use Fubini's theorem again to perform first the integral in the variable $\tau$. Taking afterwards the limit $\eta\to0^+$ we obtain the second term of the right-hand side of equation (\ref{eq:spectral_resolvent_L}).
\end{proof}

\begin{cor}
\label{thm:spectral_expansion}
For each $n\in\N$ and for each $f\in\cS$ we have two spectral expansions:
\begin{gather}
f(x) = B_{\oL}(N_n)f(x) + \int_{M_n}\sum_{j=1}^2\ef_j(x,\lambda)T_jf(\lambda)\,d\lambda, 
\label{eq:spectral_L_Mn}, \\[4pt]
f(x) = B^*_{\oL}(N_n)f(x) + \int_{M_n}\sum_{j=1}^2\aef_j(x,\lambda)S_jf(\lambda)\,d\lambda.
\label{eq:spectral_La_Mn}
\end{gather}
Moreover, the following analogues of the Parseval equality hold for $f,g\in\cS$:
\begin{gather}
(f,g) = \big(B_{\oL}(N_n)f,g\big)
+\int_{M_n}\sum_{j=1}^2T_jf(\lambda)\overline{S_jg(\lambda)}\,d\lambda, \label{eq:parseval1} \\
(f,g) = \big(B_{\oL}^*(N_n)f,g\big)
+\int_{M_n}\sum_{j=1}^2S_jf(\lambda)\overline{T_jg(\lambda)}\,d\lambda. \label{eq:parseval1_aL}
\end{gather}
\end{cor}
\begin{proof}
Apply $\oL-z\id$ to equation (\ref{eq:spectral_resolvent_L}) and $\oLa-z\id$ to equation (\ref{eq:spectral_resolvent_La}), and notice that we can differentiate four times under the integral sign since,  for $j=1,2$, $T_jf,S_jf\in L^1(M_n)$, and $\ef_j$, $\aef_j$, and their derivatives with respect to their first variable up to third order, are bounded functions on $\R\times M_n$. Equalities (\ref{eq:parseval1}) and (\ref{eq:parseval1_aL}) are obtained, respectively, by multiplying equations (\ref{eq:spectral_L_Mn}) and (\ref{eq:spectral_La_Mn}) by $\overline{g(x)}$ and integrating in $x$ over $\R$, interchanging the order of integration in the second terms, which is clearly allowed. 
\end{proof}

\section{The bounded linear maps $U_n$ and $V_n$}

Let us introduce the Hilbert space $\cJ= L^2(M)\oplus L^2(M)$, and the projections $P_j:\cJ\to L^2(M)$, $j=1,2$, onto each of the two components of $\cJ$. The scalar product of $f,g\in\cJ$ is given by
\be
\big(f,g\big)_{\cJ} =  \int_{M}\,\sum_{j=1}^{2}P_jf(\lambda)\,\overline{P_jg(\lambda)}\,\,d\lambda,
\ee
and the norm of $f\in\cJ$ is defined by $\|f\|_{\cJ}^2=(f,f)_{\cJ}$. 

Let $\chi_b:M\to\R$ be the characteristic function of the set $b\subseteq M$. For any Borel set $b\subseteq\R$ we define the orthogonal projection $\hat{\chi}(b)$ in $\cJ$ by $P_j\hat{\chi}(b)f=\chi_{M\cap b}P_jf$, for $j=1,2$, and $f\in\cJ$. For $n\in\N$ and $b\in\cB$ we also define the orthogonal projection $\hat{\chi}_n(b)=\hat{\chi}(M_n)\hat{\chi}(b)$ in $\cJ$. Notice that $\hat{\chi}_n(b)\leq \hat{\chi}_m(b)$ if $n\leq m$ and $\slim_n \hat{\chi}_n(b) = \hat{\chi}(b)$. It is convenient to introduce the notation $\cJ_n=\ran\big(\hat{\chi}_n(\R)\big)$.
Finally, we define the selfadjoint linear operator $Q$ in $\cJ$ given by
\be
\begin{gathered}
\dom(Q)=\left\{f\in\cJ,\, \, \int_{M}\lambda^2\sum_{j=1}^2\big|P_jf(\lambda)\big|^2d\lambda<\infty\right\}, \\[2pt]
P_jQ f(\lambda) = \lambda P_jf(\lambda), \quad j=1,2, \quad \lambda\in M, \quad f\in\dom(Q).
\end{gathered}
\ee
Evidently $\sigma(Q)=M$ and $\hat{\chi}$ is the resolution of the identity for $Q$. For $n\in\N$ we denote by $Q_n$ the restriction of $Q$ to $\cJ_n\cap\dom(Q)$.
Clearly, $\sigma(Q_n)=M_n$.

For $f\in\cS$ and $n\in\N$ the functions on $M$ given by  $\chi_{M_n}T_jf$ and $\chi_{M_n}S_jf$ ($j=1,2$) belong to $L^2(M)$. Therefore, each of these functions define linear transformations from $\cS$ to $L^2(M)$. We are going to show that these transformations are bounded and, therefore, can be continuously extended from $\cS$ to $\cH$. This fact allows us to define the bounded linear maps $U_n$ and $V_n$ from $\cH$ to $\cJ$ of theorem \ref{thm:TnSn_bounded} below, which is one of the key points of this work.
\begin{thm}
\label{thm:TnSn_bounded}
For each $n\in\N$ there are two bounded linear maps, $U_n:\cH\to\cJ$ and $V_n:\cH\to\cJ$, such that, for any $f\in\cS$
\be
P_jU_n f=\chi_{M_n}T_jf, \quad  P_jV_n f=\chi_{M_n}S_jf, \quad j=1,2.
\ee
\end{thm}
\begin{proof}
We proof the statement for $V_n$. The proof of the statement for $U_n$ is similar.
Along this proof we use the indices $j$, $k$, and $l$ and $m$, which take the values $1\leq j\leq 2$, $2\leq k\leq 4$, $1\leq l\leq 3$, and $2\leq m\leq 3$. We denote by $\|\cdot\|_{L_n^2}$ the norm in $L^2(M_n)$. 

Let $f\in\cS$. From the definition of $S_jf$ and lemma \ref{thm:sols_bounded}, equation (\ref{eq:sols_bounded_L}), we have
\be
S_jf = \sum_{k=2}^4 \beta_{jk}\big(g_k+u_k\big) + \sum_{l=1}^3 \alpha_{jl}\big(h_l+v_l\big),
\ee
where, for $\lambda\in M$,
\be
g_k(\lambda) = \int_{-\infty}^0 f(x)\,\overline{\exp(\iu\mu_kx)}dx, \;\;
u_k(\lambda) = \int_{-\infty}^0 f(x)\,\overline{\exp(\iu\mu_kx)\,s_k(x,\lambda)}\,dx, 
\ee
and,
\be
h_l(\lambda) = \int_0^\infty f(x)\,\overline{\exp(\iu\mu_lx)}dx, \quad
v_l(\lambda) = \int_0^\infty f(x)\,\overline{\exp(\iu\mu_lx)\,r_l(x,\lambda)}\,dx.
\ee
Recall $\mu_1=-\mu_4=\iu\mur$ and $\mu_2=-\mu_3=\nu$, where $\theta$ and $\nu$ are functions (of $\lambda$ in the previous expressions) on $[\hp,\infty)$ given by equation (\ref{eq:thetanu}). For fixed $\lambda\in M$ the functions $\exp(\iu\mu_lx)r_l(x,\lambda)$ are square integrable in $x$ over $(0,\infty)$, and the functions $\exp(\iu\mu_kx)s_k(x,\lambda)$ are square integrable in $x$ over $(-\infty,0)$. The Schwarz inequality gives
\be
|u_k(\lambda)| \leq c\lambda^{-1/4} \|f\|, \quad |v_l(\lambda)| \leq c\lambda^{-1/4} \|f\|, \quad \lambda\in M_n,
\ee
where we used the bounds on $r_k$ and $s_k$ of theorem \ref{thm:sols_edo}, possibly with a larger value of $c$. Taking into account the bounds on $\beta_{jk}$ and  $\alpha_{jl}$ of lemma \ref{thm:sols_bounded}, the restrictions of $\beta_{jk}u_k$ and $\alpha_{jl}v_l$ to $M_n$ are square integrable functions and satisfy the estimates
\be
\|\beta^*_{jk}u_k\|_{L_n^2}\leq c_n \|f\|, \quad \|\alpha^*_{jl}v_l\|_{L_n^2}\leq c_n \|f\|. 
\ee
The Schwarz inequality implies also that both $|g_4|$ and $|h_1|$ are bounded by $(2\theta)^{-1/2}\|f\|$. The bounds on $\beta_{j4}$ and  $\alpha_{j1}$ imply that the restrictions of  $\beta_{j4}g_4$ and $\alpha_{j1}h_1$ to $M_n$ are square integrable functions and satisfy the estimates
\be
\|\beta_{j4}g_4\|_{L_n^2}\leq c_n \|f\|, \quad \|\alpha_{j1}h_1\|_{L_n^2}\leq c_n \|f\|. 
\ee
To analyze $g_m$ and $h_m$ (recall $2\leq m\leq 3$) we split $f=f_++f_-$, where $f_+=f(x)$ if $x\geq0$ and $f_+(x)=0$ if $x<0$. Then $g_m(\lambda)=\tilde{f_-}\big(\mu_m(\lambda)\big)$ and $h_m(\lambda)=\tilde{f_+}\big(\mu_m(\lambda)\big)$, where $\tilde{f_+}$ and $\tilde{f_-}$ are the Fourier transforms of $f_+$ and $f_-$, respectively. For the restriction of $\alpha_{jm}h_m$ to $M_n$ we have,
\be
\|\alpha_{jm}h_m\|_{L_n^2}^2= \int_{M_n} \big|\alpha_{jm}(\lambda)\big|^2\big|\tilde{f_+}\big(\mu_m(\lambda)\big)\big|^2d\lambda.
\ee
Performing the change of variable $\lambda=\pchar(\xi)$ we obtain
\be
\|\alpha_{jm}h_m\|_{L_n^2}^2= \int_{\widehat{M}_n} \big|\alpha_{jm}\big(\pchar(\xi)\big)\big|^2\big|\tilde{f_+}\big((-1)^m\xi\big)\big|^2\pchar^\prime(\xi)d\xi,
\ee
where
$\widehat{M}_n$ is the image of $M_n$ under the inverse of $\pchar$.
The bounds on $\alpha_{jm}$ (lemma \ref{thm:sols_bounded}) imply that $\pchar^\prime(\xi)|\alpha_{jm}\big(\pchar(\xi)\big)|^2$ is bounded on $\widehat{M}_n$, so that
\be
\|\alpha_{jm}h_m\|_{L_n^2}^2\leq c_n\int_{\widehat{M}_n} \big|\tilde{f_+}\big((-1)^m\xi\big)\big|^2d\xi\leq c_n\|f_+\|^2\leq c_n\|f\|^2,
\ee
where we used the genuine Parseval equality. Thus, we get the estimate $\|\alpha_{jm}h_m\|_{L_n^2}\leq c_n\|f\|$. In a similar way we obtain $\|\beta_{jm}g_m\|_{L_n^2}\leq c_n\|f\|$. 
Putting all the pieces together, we find that the restriction of $S_jf$ to $M_n$ satisfies $\|S_jf\|_{L_n^2}\leq c_n\|f\|$.

Thus, we can define the bounded linear transformation $V_n:\cH\to\cJ$ by setting $P_jV_nf=\chi_{M_n}S_jf$ for $f\in\cS$, and the results proved above provide the bound $\|V_nf\|_{\cJ}\leq c_n\|f\|$. The image of $\cH\setminus\cS$ under $V_n$ is obtained by continuity.
\end{proof}

The analogues of Parseval's equality, equations  (\ref{eq:parseval1}) and (\ref{eq:parseval1_aL}), are extended by continuity to any two functions $f,g\in\cH$, and can be written in the concise ways
\be
(f,g) = \big(B_{\oL}(N_n)f,g\big) + \big(U_nf,V_ng\big)_{\cJ}
=\big(B^*_{\oL}(N_n)f,g\big)+\big(V_nf,U_ng\big)_{\cJ}.
\label{eq:parseval}
\ee

The next goal is to show that $\ran(U_n)$ and $\ran(V_n)$ are dense in $\cJ_n$. This is an immediate consequence, which we state in corollary \ref{thm:ran_UnVn_density}, of the technical results collected in lemmas \ref{thm:fgh}, \ref{thm:fgh_projections}, \ref{thm:ranUnVn}, and \ref{thm:lin_indep_phia}.

\begin{lem}
\label{thm:fgh}
Let $n\in\N$ and $f_1,f_2\in  L^1(M_n)\cap L^2(M_n)$, and define the functions $g$ and $h$ on $\R$ by
\be
g(x) = \int_{M_n}\sum_{j=1}^2\ef_j(x,\lambda)f_j(\lambda)\,d\lambda, \quad h(x) = \int_{M_n}\sum_{j=1}^2\aef_j(x,\lambda)f_j(\lambda)\,d\lambda, 
\label{eq:fgh}
\ee
for $x\in\R$. Then $g,h\in \cH$.
\end{lem}
\begin{proof}
Since $\ef_j$ and $\aef_j$ are continuous and bounded on $\R\times M_n$ and $f_j$ integrable over $M_n$, the functions $g$ and $h$ are well defined (and continuous) on $\R$. We prove $g\in\cH$. That $h\in\cH$ is proved similarly. 

Let us split $g=g_++g_-$, where $g_+(x)=g(x)$ if $x\geq 0$ and $g_+(x)=0$ if $x<0$. We shall proof $g_+,g_-\in\cH$. Using equation (\ref{eq:sols_bounded_L}) for $\ef_j$ we write $g_+=u_1+u_2+u_3$, where, for $1\leq k\leq 3$,
\be
u_k(x) = \int_{M_n}\exp(\iu\mu_k x)\Big(1+r_k(x,\lambda)\Big)\sum_{j=1}^2\alpha_{jk}(\lambda) f_j(\lambda)\,d\lambda, \quad x\geq 0,
\ee
and $u_k(x)=0$ if $x<0$. By lemma \ref{thm:sols_bounded} the functions $\alpha_{jk}$ are bounded on $M_n$. Since $\theta(\lambda)\geq\theta_0>0$ for all $\lambda\in M$, it is clear that $|u_1(x)|$ is bounded by a constant times $\exp(-\theta_0|x|)$ and hence $u_1\in\cH$. To analyze $u_2$ we notice that it has two contributions, one which involves the function $r_2$ and another one that does not. From the bound on $r_2$ (theorem \ref{thm:sols_edo}) it is clear that the contribution which involves $r_2$ is bounded by a constant times $1/(1+|x|)$ and thus it belongs to $\cH$. Making the change of variable $\lambda=\pchar(\xi)$ in the integral, the contribution to $u_2$ of the term that does not involve $r_2$ is given by
\be
\int_{\widehat{M}_n}\exp(\iu\xi x)\sum_{j=1}^2\alpha_{j2}\big(\pchar(\xi)\big) f_j\big(\pchar(\xi)\big)\pchar^\prime(\xi)\,d\xi, \quad x\geq 0, \label{eq:fourier}
\ee
where $\widehat{M}_n$ is the image of $M_n$ by the inverse of $\pchar$. We prove below that the above integral is the Fourier transform of a square integrable function on $\R$, and thus it belongs to $\cH$, what implies $u_2\in\cH$. To see that (\ref{eq:fourier}) is the Fourier transform of a square integrable function on $\R$, notice that, for $j=1,2$,
\be
\begin{gathered}
\int_{\widehat{M}_n}\!\big|\alpha_{j2}\big(\pchar(\xi)\big) f_j\big(\pchar(\xi)\big)\big|^2\pchar^{\prime\,2}(\xi)\,d\xi =\! \int_{M_n}\!\pchar^{\prime}\big(\nu(\lambda)\big)\big|\alpha_{j2}(\lambda)\big|^2\big| f_j(\lambda)\big|^2\,d\lambda \\[4pt]
\leq c  \int_{M_n}\big| f_j(\lambda)\big|^2\,d\lambda < \infty,
\end{gathered}
\ee
where we used that the bounds (\ref{eq:bounds_alpha_beta}) imply that $\pchar^\prime(\nu)|\alpha_{jk}|^2$ is bounded on $M_n$. The proof that $u_3\in\cH$ is similar. Thus, $g_+\in\cH$. In the same way we can proof that $g_-\in\cH$, and therefore we obtain $g\in\cH$. 
\end{proof}

\begin{lem}
\label{thm:fgh_projections}
Let $n$, $f_1$, $f_2$, $g$ and $h$ as in lemma \ref{thm:fgh}, and $G$ a bounded open subset of $\R$.
Then:
\begin{gather}
B_{\oL}(N_n)g=0, \quad B_{\oL}^*(N_n)h=0, \label{eq:BLNng} \\[4pt]
B_{\oL}(G) g(x) = \int_{M_n\cap G} \sum_{j=1}^2 \ef_j(x,\lambda)f_j(\lambda)\,d\lambda, \quad x\in\R,
\label{eq:f_integral_projection} \\[4pt]
B^*_{\oL}(G) h(x) = \int_{M_n\cap G} \sum_{j=1}^2 \aef_j(x,\lambda)f_j(\lambda)\,d\lambda,  \quad x\in\R.
\end{gather} 
\end{lem}
\begin{proof}
We prove the statements for $g$. The statements for $h$ are proved similarly.
We start with the proof of equation (\ref{eq:f_integral_projection}), which relies on the analogue of the Stone formula, equation (\ref{eq:stone}). For $z\in\rho(\oL)$ it is clear that $R_{\oL}(z)g$ is given by substituting $f_j(\lambda)$ by $f_j(\lambda)/(\lambda-z)$  in the expression that defines $g$.
Hence, for $\xi\in\R$ and $\epsilon>0$ we have
\be
R_{\oL}(\xi+\iu\epsilon)g(x)-R_{\oL}(\xi-\iu\epsilon)g(x)=\!\!\int_{M_n}\!\frac{2\iu\epsilon}{(\lambda-\xi)^2+\epsilon^2}\sum_{j=1}^2\ef_j(x,\lambda)f_j(\lambda)\,d\lambda.
\ee
Let $(a,b)$ be an open bounded interval. Inserting the above equation into equation (\ref{eq:stone}),
interchanging the order of integration in the iterated integral (Fubini's theorem is clearly applicable), and manipulating the resultant expression in a standard way we get
\be
\frac{1}{2}\Big(B_{\oL}\big((a,b)\big)+B_{\oL}\big([a,b]\big)\Big)g(x) = 
\int_{M_n\cap(a,b)}\sum_{j=1}^2\ef_j(x,\lambda)f_j(\lambda)\,d\lambda\,d\xi.
\label{eq:stone_g}
\ee
Relation (\ref{eq:stone_g}) shows that if $\lambda_e$ is an eigenvalue of $\oL$ then $B_{\oL}\big(\{\lambda_e\}\big)g=0$, since $\lambda_e$ is contained in an open interval $(\alpha,\beta)\subset N_n$, and then
\be
B_{\oL}\big(\{\lambda_e\}\big)g= B_{\oL}\big(\{\lambda_e\}\big)\frac{1}{2}\Big(B_{\oL}\big((\alpha,\beta)\big)+B_{\oL}\big([\alpha,\beta]\big)\Big)g = 0,
\ee
because $(\alpha,\beta)\cap M_n=\emptyset$. Hence, the left-hand side of equation (\ref{eq:stone_g}) can be replaced by $B_{\oL}\big((a,b)\big)g(x)$ and then equation  (\ref{eq:f_integral_projection}) holds if $G$ is a bounded open interval. Since any bounded open set is the union of an at most countable family of disjoint bounded open intervals, equation (\ref{eq:f_integral_projection}) holds for any open bounded set $G$.

Now, since $N_n\subset (\lambda_0,\lambda_s)$ we have 
$
B_{\oL}(N_n)g=B_{\oL}\big((\lambda_0,\lambda_s)\big)g-B_{\oL}\big((\lambda_0,\lambda_s)\cap N_n^c\big)g
$.
Applying equation (\ref{eq:f_integral_projection}) to the right-hand side of this equation we obtain $B_{\oL}(N_n)g=0$, which is the first of equations (\ref{eq:BLNng}).
\end{proof}

If $f\in\cS$ and $G$ is a bounded open subset of $M_n$, then the spectral resolutions (\ref{eq:spectral_L_Mn}) and (\ref{eq:spectral_La_Mn}), the definition of $U_n$ and $V_n$, and lemma \ref{thm:fgh_projections} imply
\begin{gather}
B_{\oL}(G)f(x) = \int_G \sum_{j=1}^2\ef_j(x,\lambda)P_jU_nf(\lambda)\,d\lambda, 
\\[4pt] 
B^*_{\oL}(G)f(x) = \int_G \sum_{j=1}^2\aef_j(x,\lambda)P_jV_nf(\lambda)\,d\lambda.
\end{gather}

\begin{lem}
\label{thm:ranUnVn}
Let $n\in\N$ and $f_1,f_2\in C(M_n)\cap L^1(M_n)\cap L^2(M_n)$, and define $g$ and $h$ as in lemma \ref{thm:fgh}. Then $U_ng =V_nh=f$, where, for $j=1,2$, $P_jf(\lambda)=f_j(\lambda)$ if $\lambda\in M_n$ and  $P_jf(\lambda)=0$ if $\lambda\in M\setminus M_n$.
\end{lem}
\begin{proof}
Consider a sequence $g_m\in\cS$, $m\in\N$, that converges to $g$ in $\cH$. Then $U_ng=\lim_m U_ng_m$. Using the spectral resolution (\ref{eq:spectral_L_Mn}) for $g_m$ and taking into account that $B_{\oL}(N_n)g=0$ (lemma \ref{thm:fgh}) we get $\lim_m v_m=0$, where 
\be
v_m(x) = \int_{M_n} \sum_{j=1}^2\ef_j(x,\lambda)\Big(f_j(\lambda)-P_jU_ng_m(\lambda)\Big)\,d\lambda, \quad x\in\R.
\ee 
Take a bounded interval $(a,b)\subset M_n$. The continuity of $B_{\oL}\big((a,b)\big)$ and lemma \ref{thm:fgh} imply
\be
\lim_{m\to\infty}\int_a^b \sum_{j=1}^2\ef_j(x,\lambda)\Big(f_j(\lambda)-P_jU_ng_m(\lambda)\Big)\,d\lambda=0, \quad  a.a.\,\, x\in\R.
\label{eq:fjmgj}
\ee
Since for each fixed $x$ the functions $\ef_j(x,\cdot)$ are square integrable on $(a,b)$, the convergence of $U_ng_m$ to $U_ng$ in $L^2(M_n)$ gives
\be
\int_a^b \sum_{j=1}^2\ef_j(x,\lambda)\Big(f_j(\lambda)-P_jU_ng(\lambda)\Big)\,d\lambda=0, \quad a.a.\,\, x\in\R.
\label{eq:fjmgj2}
\ee
Notice that $P_jU_ng$ is integrable on $(a,b)$ since it is square integrable. Therefore, the left-hand side of the above equation defines a continuous function of $x$ that has to vanish for all $x\in\R$ since it vanishes for \textit{a.a} $x\in\R$.
Since equation (\ref{eq:fjmgj2}) holds for every bounded interval contained in the open set $M_n$, the integrand has to vanish \textit{a.e.} in $M_n$ for each fixed $x\in\R$:
\be
\sum_{j=1}^2\ef_j(x,\lambda)\Big(f_j(\lambda)-P_jU_ng(\lambda)\Big)=0, \quad x\in\R, \quad a.a.\,\, \lambda\in M_n.
\ee
The linear independence of $\ef_1(\cdot,\lambda)$ and $\ef_2(\cdot,\lambda)$ requires $P_jU_ng(\lambda)=f_j(\lambda)$ for \textit{a.a.} $\lambda\in M_n$, that is, $U_ng=f$.

The proof of $U_ng=f$ is basically identical, but it requires that $\aef_1(\cdot,\lambda)$ and $\aef_2(\cdot,\lambda)$ be linearly independent for \textit{a.a.} $\lambda\in M_n$. This is the thesis of lemma \ref{thm:lin_indep_phia}, the proof of which uses the part of this lemma that has been proved.
\end{proof}

\begin{lem}
\label{thm:lin_indep_phia}
The functions $\aef_1(\cdot,\lambda)$ and $\aef_2(\cdot,\lambda)$ are linearly independent for almost all $\lambda\in M$.
\end{lem}
\begin{proof}
Along this proof the index $j$ takes the values 1 and 2. We proceed by contradiction. If the thesis is false, there are functions $c_j$ on $M$, which vanish simultaneously on at most a set of zero Lebesgue measure, and such that $\sum_jc_j(\lambda)\aef_j(\cdot,\lambda)=0$ for \textit{a.a.} $\lambda\in M$. Therefore, for any $n\in\N$ and any $h\in\cS$ we have $\sum_jc_j(\lambda)P_jU_nh(\lambda)=0$ for \textit{a.a.} $\lambda\in M$. Take functions $f_j$ as in lemma \ref{thm:ranUnVn}, so that there is a sequence of functions $U_ng_m$,  $m\in\N$, where $g_m\in\cS$, such that $P_jU_ng_m(\lambda)\to f_j(\lambda)$ for \textit{a.a} $\lambda\in M_N$. We have $\sum_j c_j(\lambda)P_jU_ng_m(\lambda)=0$, and therefore $\sum_jc_j(\lambda)f_j(\lambda)=0$, for \textit{a.a.} $\lambda\in M_n$. This property has to hold for all continuous functions $f_1$ and $f_2$ on $M_n$ that are in $L^1(M_n)\cap L^2(M_n)$, and this is only possible if $c_1$ and $c_2$ vanish \textit{a.e.} in $M_n$. Since $n$ is arbitrary, $c_1$ and $c_2$ do vanish \textit{a.e.} in $M$, what contradicts our initial assumption. Hence, $\aef_1(\cdot,\lambda)$ and $\aef_2(\cdot,\lambda)$ have to be linearly independent for \textit{a.a} $\lambda\in M$. 
\end{proof}

The following corollary is an immediate consequence of lemma \ref{thm:ranUnVn}, since $C(M_n)\cap L^1(M_n)\cap L^2(M_n)$ is dense in $L^2(M_n)$.

\begin{cor}
\label{thm:ran_UnVn_density} 
$\ran(U_n)$ and $\ran(V_n)$ are dense in $\cJ_n$.
\end{cor}
Next theorem collects some important properties of $U_n$ and $V_n$ that will be used later.

\vfill\eject

\begin{thm}
\label{thm:UV}
For all $n\in\N$ the linear maps $U_n$ and $V_n$ satisfy:
\begin{enumerate}
\item $\iker(U_n)=\ran\big(B_{\oL}(N_n)\big)$ and $\iker(V_n)=\ran\big(B^*_{\oL}(N_n)\big)$.
\item 
$U_nR_{\oL}(z) = R_{Q_n}(z) U_n$ and $V_nR_{\oLa}(z) = R_{Q_n}(z)V_n$ for all $z\in\rho(\oL)$.
\item $U_n\oL f = Q_nU_nf$ if $f\in\dom(\oL)$, and $V_n\oLa f = Q_nV_nf$ if $f\in\dom(\oLa)$.
\end{enumerate}
\end{thm}
\begin{proof}
Point 1 follows almost immediately from the analogue of the Parseval equality (\ref{eq:parseval}) and the density of $\ran(V_n)$ in $\cJ_n$.
To prove the first relation of point 2 we notice it holds for $f\in\cS$ due to equation (\ref{eq:spectral_resolvent_L}) of lemma \ref{thm:spectral_resolvent}, equation (\ref{eq:BL_R}),  point 1 of this lemma, and lemma \ref{thm:ranUnVn}. Then, the relation is extended to $\cH$ by continuity. The second relation of point 2 follows in a similar way.
The first relation of point 3 is obtained easily from point 2. 
\end{proof}

The projections $B_{\oL}(N_n)$ and $B^*_{\oL}(N_n)$ induce decompositions of $\cH$ as a direct sum of two closed subspaces:
\be
\cH = \ran\big(B_{\oL}(N_n)\big) \oplus \cU_n = \ran\big(B^*_{\oL}(N_n)\big) \oplus \cV_n,
\ee
where we introduce the notation $\cU_n=\iker\big(B_{\oL}(N_n)\big)$ and $\cV_n=\iker\big(B^*_{\oL}(N_n)\big)$. Let $\ures_n$ be the restriction of $U_n$ to $\cU_n$ and $\vres_n$ the restriction of $V_n$ to $\cV_n$.
 
\begin{thm}
\label{thm:UVrestricted}
For each $n\in\N$, $\ures_n$ is a continuous linear bijection from $\cU_n$ onto $\cJ_n$ which has a continuous inverse, and $\vres_n$ is a continuous linear bijection from $\cV_n$ onto $\cJ_n$ which has a continuous inverse.
\end{thm}
\begin{proof}
Clearly, $\ures_n$ is injective and bounded, and $\ran(\ures_n)=\ran(U_n)$. Let us show that the inverse of $\ures_n$ is bounded.
If $f\in\ran(U_n)$, then there is a unique $g\in\cU_n$ such that $U_ng=f$, and $g=\ures_n^{-1}f$. Since $g\in\cU_n$ implies $B_{\oL}(N_n)g=0$, the Parseval equality (\ref{eq:parseval}) gives $\|g\|^2=(V_ng,U_ng)_{\cJ}$, and since $V_n$ is bounded we get $\|g\|\leq\|V_n\|\|U_ng\|_{\cJ}$, that is, $\|\ures_n^{-1}f\|\leq\|V_n\|\|f\|_{\cJ}$. Since $\ures_n$ has a bounded inverse, its range of is closed, and then $\ran(\ures_n)=\cJ_n$, because, by corollary \ref{thm:ran_UnVn_density}, $\ran(\ures_n)$ is dense in $\cJ_n$. This proves that $\ures_n$ is a bounded bijection of $\cU_n$ onto $\cJ_n$ which has a bounded inverse. The claim for $\vres_n$ is proved similarly.
\end{proof}

An important byproduct of theorem \ref{thm:UVrestricted} is emphasized in the following corollary.

\begin{cor}
$\ran(U_n)=\ran(V_n)=\cJ_n$ for all $n\in\N$.
\end{cor}

The extension to $\cH$ of the spectral expansions (\ref{eq:spectral_L_Mn}) and (\ref{eq:spectral_La_Mn}) can be written concisely as
\be
\id = B_{\oL}(N_n) + \ures_n^{-1}U_n = B^*_{\oL}(N_n) + \vres_n^{-1}V_n.
\label{eq:spectral_UnVn}
\ee

\section{Spectral resolution of the identity for $\oL$}

The results of the previous sections allow us to show that there are spectral resolutions of the identity for $\oL$ and $\oLa$, and that both $\oL$ and $\oLa$ are spectral operators of scalar type.
We start noticing that for any $n\in\N$ and $b\in\cB$ the linear operators $\ures_n^{-1}\hat{\chi}_n(b)U_n$ and $\vres_n^{-1}\hat{\chi}_n(b)V_n$ are bounded projections on $\cH$. From lemma \ref{thm:fgh_projections} we obtain
\begin{gather}
\ures_n^{-1}\hat{\chi}_n(b)U_nf(x) = \int_{M_n\cap\,b} \sum_{j=1}^2\ef_j(x,\lambda) P_jU_nf(\lambda)\,d\lambda, \label{eq:Uspectral} \\[2pt]
\vres_n^{-1}\hat{\chi}_n(b)V_nf(x) = \int_{M_n\cap\,b} \sum_{j=1}^2\aef_j(x,\lambda) P_jV_nf(\lambda)\,d\lambda,
\end{gather}
for any $f\in\cS$ and $x\in\R$.

\begin{lem}
\label{thm:BLUn}
If $n\in\N$, $b_1\in\cB_b$, and $b_2\in\cB$, then
\begin{gather}
B_{\oL}(b_1)\ures_n^{-1} \hat{\chi}_n(b_2)U_n = \ures_n^{-1}\hat{\chi}_n(b_1\cap b_2)U_n,  \label{eq:projections_on_cU}
\\[4pt] 
B_{\oL}^*(b_1)\vres_n^{-1} \hat{\chi}_n(b_2)V_n = \vres_n^{-1}\hat{\chi}_n(b_1\cap b_2)V_n.
\end{gather}
\end{lem}
\begin{proof}
We prove only equality (\ref{eq:projections_on_cU}). The other equality is proved similarly. If $b_1$ is open, equation (\ref{eq:Uspectral}) and  lemma \ref{thm:fgh_projections} show that (\ref{eq:projections_on_cU}) holds in $\cS$ and then, by continuity, in $\cH$.
Consider now a general bounded Borel set $b_1$ and take an open bounded set $G_1$ such that $b_1\subset G_1$. All projections $B_{\oL}(b)$, where $b$ is a Borel set contained in $G_1$, are uniformly bounded: there is $c>0$ such that $\|B_{\oL}(b)\|<c$ (section \ref{sec:projections_bounded}). Let $\mu_{\ell}$ be the Lebesgue measure on $\R$. There is an open set $G_2$ such that $b_1\subseteq G_2\subset G_1$ and $\mu_{\ell}(G_2\setminus b_1)$ is as small as desired. Let $f\in\cH$ and $g=\ures_n^{-1}\hat{\chi}_n(b_2)U_nf$. Since $G_2$ is open and bounded we have $B_{\oL}(G_2)g=\ures_n^{-1}\hat{\chi}_n(G_2\cap b_2)U_nf$, and thus
\be
B_{\oL}(b_1)g-\ures_n^{-1}\hat{\chi}_n(b_1\cap b_2)U_nf = \ures_n^{-1}\hat{\chi}_n\big((G_2\setminus b_1)\cap b_2\big) U_nf - B_{\oL}(G_2\setminus b_1)g.
\label{eq:small_difference}
\ee
There is also an open set $G_3$ such that $G_2\setminus b_1\subseteq G_3$ and $\mu_{\ell}(G_3)-\mu_{\ell}(G_2\setminus b_1)$ is as small as desired. We have
\be
\begin{gathered}
\|B_{\oL}(G_2\setminus b_1)g\| = \|B_{\oL}(G_2\setminus b_1)B_{\oL}(G_3)g\| \leq c \|B_{\oL}(G_3)g\|  \\[4pt]
\leq c \|\ures_n^{-1}\|\|\hat{\chi}_n(G_3\cap b_2)U_nf\|_{J_n},
\end{gathered}
\ee
where we used equality (\ref{eq:projections_on_cU}) for $G_3$, which is open and bounded.
Take any $\epsilon>0$. By a proper choice of $G_2$  and $G_3$ we can have
\begin{gather}
\|\hat{\chi}_n\big((G_2\setminus b_1)\cap b_2\big)U_nf\|_{J_n} \leq \frac{\epsilon}{2\|\ures_n^{-1}\|}, \\[4pt]
\|\hat{\chi}_n(G_3\cap b_2)U_nf\|_{J_n} \leq \frac{\epsilon}{2c\|\ures_n^{-1}\|}.
\end{gather}
Then, from equation (\ref{eq:small_difference}) we obtain $\|B_{\oL}(b_1)g-\ures_n^{-1}\hat{\chi}(b_1\cap b_2)U_nf \|\leq\epsilon$ for any $\epsilon>0$, what proves equality (\ref{eq:projections_on_cU}).
\end{proof}

We  extend  $B_{\oL}$ to a bounded algebra homomorphism, $E_{\oL}$, from $\cB$ onto a subset of $\cL(\cH)$,  by defining the map $E_{\oL}:\cB\to\cL(\cH)$ by
\be
E_{\oL}(b) = B_{\oL}(N_n\cap b) + \ures_n^{-1}\hat{\chi}_n(b)U_n, \quad b\in\cB, \quad n\in\N.
\ee
The right-hand side of the above equation is independent of $n$, what can be easily proved using the following equalities, which hold for $n\geq m$:
first, from the definitions it is obvious that $U_m=\hat{\chi}_n(M_m)U_n$ and $V_m=\hat{\chi}_n(M_m)V_n$; second, $\ures_n^{-1}\hat{\chi}_n(M_m)\hat{\chi}_n(b)U_n=\ures_m^{-1}\hat{\chi}_m(b)U_m$ (this is also rather obvious but it can be proved using the analogue of the Parseval equality); 
and third, $N_n\cup (M_n\setminus M_m)=N_m$. From these facts, and using lemma \ref{thm:BLUn}, we have
\be
\begin{gathered}
B_{\oL}(N_n\cap b) + \ures_n^{-1}\hat{\chi}_n(b)U_n =  B_{\oL}(N_n\cap b) + \ures_n^{-1}\hat{\chi}_n(M_1)\hat{\chi}_n(b)U_n 
 \\[2pt]
+ \ures_n^{-1}\hat{\chi}_n(M_n\setminus M_1)\hat{\chi}_n(b)U_n  = B_{\oL}(N_n\cap b) +  \ures_1^{-1}\hat{\chi}_1(b)U_1
 \\[2pt]
 + B_{\oL}\big((M_n\setminus M_1)\cap b\big) \ures_n^{-1}U_n = B_{\oL}(N_1\cap b) + \ures_1^{-1}\hat{\chi}_1(b)U_1.
\end{gathered}
\ee

The strong limit of $ \ures_n^{-1}\hat{\chi}_n(b)U_n$ as $n\to\infty$ exits,
\be
 \slim_{n\to\infty} \ures_n^{-1}\hat{\chi}_n(b)U_n = B_{\oL}\big((M\setminus M_1)\cap b\big) + \ures_1^{-1}\hat{\chi}_1(b)U_1, \label{eq:strong_limit}
\ee
and is a bounded projection, with a bound independent of $b$. Then, to fully exploit the spectral expansions, equations (\ref{eq:spectral_L_Mn}) and (\ref{eq:Uspectral}), it may be interesting to express $E_{\oL}$ as
\be
E_{\oL}(b) = B_{\oL}(N\cap b) + \slim_{n\to\infty} \ures_n^{-1}\hat{\chi}_n(b)U_n, \quad b\in\cB.
\ee

We are ready to discuss the conclusions of this work, collected in the next two theorems.

\begin{thm}
\label{thm:spectral_measure}
The map $E_{\oL}$ extends $B_{\oL}$ from $\cB_b$ to $\cB$, and is a spectral resolution of the identity for $\oL$. That is:
\begin{enumerate}
\item  $E_{\oL}(b)=B_{\oL}(b)$ for each bounded Borel set $b$.
\item There is $c>0$ such that $\|E_{\oL}(b)\|\leq c$ for all $b\in\cB$.
\item $E_{\oL}(\emptyset)=0$ and $E_{\oL}(\R)=E_{\oL}\big(\sigma(\oL)\big)=\id$.
\item $E_{\oL}(b_1\cap b_2)=E_{\oL}(b_1)E_{\oL}(b_2)$ for $b_1,b_2\in\cB$.
\item $E_{\oL}(b_1\cup b_2)=E_{\oL}(b_1)+E_{\oL}(b_2)-E_{\oL}(b_1)E_{\oL}(b_2)$ for $b_1,b_2\in\cB$.
\item If $\{b_i, i\in\N\}$ is a sequence of pairwise disjoint Borel sets and $f\in\cH$, then $E_{\oL}\big(\cup_i b_i\big) f = \sum_{i=1}^\infty E_{\oL}(b_i)f$.
\item If $\{b_i, i\in\N\}$ is an ascending sequence of Borel sets and $f\in\cH$, then \mbox{$E_{\oL}\big(\cup_i b_i\big) f = \lim_i E_{\oL}(b_i)f$}. If the sequence is descending, then we have $E_{\oL}\big(\cap_i b_i\big) f = \lim_i E_{\oL}(b_i)f$.  
\item If $b\in\cB$ and $z\in\rho(\oL)$ then $E_{\oL}(b) R_{\oL}(z) = R_{\oL}(z)E_{\oL}(b)$.
\item If $b\in\cB$ is bounded, then $\ran\big(E_{\oL}(b)\big)\subset\dom(L)$.
\item If $b\in\cB$ and $f\in\dom(\oL)$ then  $E_{\oL}(b)f\in\dom(\oL)$ and
\be
E_{\oL}(b)\oL f=\oL E_{\oL}(b)f.
\ee
\item If $b\in\cB$ and $\oL_b$ is the restriction of $\oL$ to $\ran\big( E_{\oL}(b)\big)$, then $\sigma(\oL_b)\subseteq\overline{b}$.
\end{enumerate}
\end{thm}
\begin{proof}
Point 1 follows easily from lemma \ref{thm:BLUn} (with $b_1=b$ and $b_2=\R$) and equality (\ref{eq:spectral_UnVn}), since for any $b\in\cB_b$ we have 
\be
\begin{gathered}
E_{\oL}(b) = B_{\oL}(b)  \Big(B_{\oL}(N_n) + \ures_n^{-1}U_n\Big) = B_{\oL}(b).
\end{gathered}
\label{eq:restriction_for_bounded_open_sets}
\ee

Point 2 follows from the fact that the projections $B_{\oL}(b)$ corresponding to Borel sets $b$ contained in the bounded set $N_1$ are uniformly bounded, and  that $\|\hat{\chi}_1(b)\|\leq 1$ for all $b\in\cB$. 
Point 3 follows from $B_{\oL}(\emptyset)=0$ and $\hat{\chi}(\emptyset)=0$, from $\sigma(\oL)=M_n\cup N_n$, and from relation (\ref{eq:spectral_UnVn}). 
Points 4, 5, and 6 follow from the corresponding properties of the projections $B_{\oL}$ and $\hat{\chi}_n$, and from $B_{\oL}(N_n\cap b_1)\ures_n^{-1}\hat{\chi}_n(b_2)U_n=0$ for any $b_1,b_2\in\cB$. 
Point 7 follows from point 6 and simple standard arguments.
Point 8 follows from the corresponding property for $B_{\oL}$, equation (\ref{eq:BL_R}), and point 2 of theorem \ref{thm:UV}.
Point 9 follows from equality (\ref{eq:projection_L_bounded}), since $E_{\oL}(b)=B_{\oL}(b)$ for $b\in\cB_b$.
Point 10 follows easily from point 9 of this theorem.

To prove point 11, we determine $\rho(\oL_b)$ by considering $(\oL_b-z\id)f=g$ for $g\in\ran\big(E_{\oL}(b)\big)$ and $z\in\C$. Applying $B_{\oL}(N_n\cap b)$ and $\hat{\chi}_n(b)U_n$ to both sides of this equation, and using point 3 of theorem \ref{thm:UV}, we get
\begin{gather}
E_{\Lb}(N_n\cap b) (\Lb-z\id)\ioLth f =  E_{\Lb}(N_n\cap b) \ioLth g, \label{eq:resolvent_restriction_1}\\[4pt]
\big(Q_n-z\id)\hat{\chi}_n(b)U_nf = \hat{\chi}_n(b) U_ng. \label{eq:resolvent_restriction_2}
\end{gather}
We will have $z\in\rho(\oL_b)$ if and only if the above equations have a unique uniformly bounded solution for each $g\in\ran\big(E_{\oL}(b)\big)$. If $h\in\ran\big(E_{\Lb}(N_n\cap b)\big)$ then $h\in\dom(\Lb)\subset\dom(\oLth)$ and
$h =\ioLth E_{\oL}(N_n\cap b)\oLth h$. Therefore, equation (\ref{eq:resolvent_restriction_1}) has a unique uniformly bounded solution if and only if $z\in\rho(\Lb_{N_n\cap b})$, where $\Lb_{N_n\cap b}$ is the restriction of $\Lb$ to $\ran\big(E_{\Lb}(N_n\cap b)\big)$. Since the image of $\ran\big(E_{\oL}(b)\big)$ under $\hat{\chi}_n(b)U_n$ is equal to $\ran\big(\hat{\chi}_n(b)\big)$, equation (\ref{eq:resolvent_restriction_2}) will have a unique uniformly bounded solution if and only if $z\in\rho(Q_{M_n\cap b})$, where $Q_{M_n\cap b}$ is the restriction of $Q_n$ to $\ran\big(\hat{\chi}(M_n\cap b)\big)$. Therefore, $\rho(\oL_b)=\rho(\Lb_{N_n\cap b})\cap\rho(Q_{M_n\cap b})$, and then $\sigma(\oL_b) =\sigma(\Lb_{N_n\cap b})\cup\sigma(Q_{M_n\cap b})\subseteq (\overline{N_n\cap b})\cup (\overline{M_n\cap b})\subseteq\overline{b}$, where we used the spectral properties of the selfadjoint operators $\Lb$ and $Q$.
\end{proof}

The integral with respect to a projection valued measure such as $E_{\oL}$ or $E_{\Lb}$, which appears in the next theorem, is defined and studied, for instance, in Dunford and Schwartz, chapter XI  \cite{Dunford1971}.

\begin{thm}
\label{thm:spectral_operator}
If $\{b_n$, $n\in\N\}$ is an ascending sequence of bounded Borel sets such that $\sigma(\oL)\subseteq \cup_n b_n$, then, for $f\in\dom(\oL)$,
\be
\oL f= \lim_{n\to\infty} \int_{b_n} \lambda \, E_{\oL}(d\lambda)\,f. \label{eq:L_spectral}
\ee
\end{thm}
\begin{proof}
We notice that $E_{\oL}(b_n)=B_{\oL}(b_n)$ since $b_n$ is bounded for all $n\in\N$. For any $b\subseteq b_n$ we can write equation (\ref{eq:projection_L_bounded}) as
\be
B_{\oL}(b) = (\oLo+P_0)^{-1}\ioLth \Lb E_{\Lb}(b_n) E_{\Lb}(b)\ioLth + P_0\oLth E_{\Lb}(b)\ioLth.
\ee
Then, since $\Lb E_{\Lb}(b_n)$ is a bounded operator, we have
\be
\begin{gathered}
\int_{b_n} \lambda \, E_{\oL}(d\lambda) = 
 (\oLo+P_0)^{-1}\ioLth \Lb E_{\Lb}(b_n) \int_{b_n} \lambda \, E_{\Lb}(d\lambda)\, \ioLth 
 \\[2pt]
 + P_0\oLth\int_{b_n} \lambda \, E_{\Lb}(d\lambda)\, \ioLth = \oL B_{\oL}(b_n) = \oL E_{\oL}(b_n) ,
 \end{gathered}
 \label{eq:spectral_operator_bounded} 
\ee
where in the last equality we used $\int_{b_n}\lambda E_{\Lb}(d\lambda)=\Lb E_{\Lb}(b_n)$.  For $f\in\dom(\oL)$ we have $\oL E_{\oL}(b_n)f = E_{\oL}(b_n)\oL f$  (point 10 of theorem \ref{thm:spectral_measure}). Using equation (\ref{eq:spectral_operator_bounded}) and taking the limit $n\to\infty$  (using points 7 and 3 of theorem \ref{thm:spectral_measure}) we obtain Equation  (\ref{eq:L_spectral}).
\end{proof}

Theorem \ref{thm:spectral_operator} shows that $\oL$ is an unbounded spectral operator of scalar type, in the sense of Bade \cite{Bade1954}. Furthermore, theorems \ref{thm:spectral_measure} and  \ref{thm:spectral_operator} ensure that it is possible to establish an operational calculus for $\oL$ as it is done for selfadjoint operators. It is clear that $\oLa$ is also an unbounded spectral operator of scalar type, and its spectral resolution of the identity is $E_{\oLa}(b)=E_{\oL}(b)^*$ for $b\in\cB$.

\section{Final remarks}

If, instead of $\int_\R (1+|x|^3)|\pot^{(k)}_i(x)|dx<\infty$, we require that the functions $\pot_i$ satisfy $\int_\R \exp(\gamma|x|)|\pot^{(k)}_i(x)|dx<\infty$ for some $\gamma>0$, $i=1,2$, and $0\leq k\leq 2$, as it happens in the case of the spin wave dynamics in presence of a magnetic soliton, it can be shown that the fundamental matrices of theorem \ref{thm:sol_system} are such that,  for fixed $x$, $\Phi_+(x,\cdot)$ and $\Phi_-(x,\cdot)$ can both be extended analytically from the upper and lower half-planes to open sets that contain $\R\setminus\{\hm,\hp\}$. This implies that the zeros of $W_+$ and $W_-$ are isolated and can accumulate only at $\hm$ and $\hp$. Therefore, in this case $N=\overline{\sigma_p(\oL)}\cup\{\hp\}$ and $M=(\hp,\infty)\setminus\sigma_p(\oL)$, and we have, for any $f\in\cH$,
\be
B_{\oL}(N) f = \sum_{n} \sum_{\alpha=1}^{d_n} (f,\xi_{n\alpha}^*)\xi_{n\alpha},
\ee
where $n$ runs over the eigenvalues of $\oL$ and $d_n\leq 2$ is the dimension of the fundamental subspace corresponding to the $n$-th eigenvalue, $\lambda_n$. For each $n$, the functions $\xi_{n\alpha}$ span the fundamental subspace of $\oL$ corresponding to $\lambda_n$, and the functions $\xi_{n\alpha}^*$ span the  fundamental subspace of $\oLa$ corresponding to $\lambda_n$, and satisfy the condition $(\xi_{n\alpha},\xi^*_{m\beta})=\delta_{nm}\delta_{\alpha\beta}$. If $\psi_{n\alpha}$ constitute an orthogonal basis the fundamental subspace of $\Lb$ corresponding to $\lambda_n$, we can take $\xi_{n\alpha}=\oLth\psi_{n\alpha}$ and $\xi^*_{n\alpha}=\ioLth\psi_{n\alpha}$.

The main obstacle to proof that an operator is spectral is found in proving that it has no spectral singularity (that is, that the spectral projections are uniformly bounded). For instance, Huige proved \cite{Huige1971} that ordinary differential operators of a certain class, which include $\oL$ and $\oLa$, have, under certain hypothesis, spectral projections associated to the Borel sets whose closure exclude a finite set of exceptional points. However, these projections need not be uniformly bounded. Moreover, among the hypothesis there is one which implies the assumption that the only spectral singularities belong to the exceptional set.

Mackey proved (theorem 55 of \cite{Mackey1952}, see also \cite{Wermer1954}) that any bounded spectral measure, such as $E_{\oL}$, is similar to a selfadjoint spectral measure, in the sense that there is a continuous linear bijection $A$ of $\cH$, hence with continuous inverse, such that $A^{-1}E_{\oL}(b)A$ is selfadjoint for each $b\in\cB$. This means that any spectral operator of scalar type is similar to a self-adjoint operator. In our case, there is a selfadjoint operator $H$ such that $A^{-1}\oL A=H$. A similar statement is true for $\oLa$. Evidently, $A$ is not equal to $\oLth$ and thus $H$ is not equal to $\Lb$. Nonselfadjoint operators that are similar to selfadjoint operators are sometimes called quasi-selfadjoint operators in the literature \cite{Bagarello2015}. Actually, quasi-selfadjoint operators are defined as operators which are similar to their adjoints, but it is easy to show (proposition 5.5.2 of \cite{Bagarello2015}) that quasi-selfadjointness is equivalent to being similar to a selfadjoint operator. Thus, using this terminology, the conclusion of this work may be quickly summarized by saying that $\oL$ and $\oLa$ are quasi-selfadjoint operators.

\section*{Acknowledgements}
This work received financial support from Grant No. PID2022-138492NB-I00-XM4, funded by MCIN/AEI/10.13039/501100011033, and from Grant No. E11\_23R/M4, funded by Diputaci\'on General de Arag\'on.


\appendix

\section{Proof of theorem \ref{thm:sols_edo} \label{app:sols_edo}}

To proof theorem \ref{thm:sols_edo} we analyze the solutions of $\oL u = \lambda u$ for $\lambda\in [\hp,\lambda_s]$ and for $\lambda\in[\lambda_s,\infty)$ separately. Recall that $\lambda_s>\hp$ is a sufficiently large positive number. This fourth order differential equation is equivalent to the first order linear system (\ref{eq:system}).
Recall also that for $\lambda\in [\hp,\infty)$ the roots of $\pchar(\mu)-\lambda$ 
are $\mu_1 = \iu\mur$, $\mu_2 = \nu$, $\mu_3=-\nu$, $\mu_4=-\iu\mur$, with $\theta$ and $\nu$ functions of $\lambda$ given by equations (\ref{eq:thetanu}). We denote by $e_j$ ($1\leq j\leq 4$) the vectors of the canonical basis of $\C^4$. For $1\leq k\leq 4$ the eigenvalues of the matrix $A$ entering equation (\ref{eq:system}) are $\iu\mu_k$, and an eigenvector corresponding to $\iu\mu_k$ is $p_k=\sum_j (\iu\mu_k)^{j-1}e_j$. The $4\times 4$ matrix whose $k$-th column is $p_k$ is denoted by $\Pi$. If all the eigenvalues of $A$ are different, then $\Pi$ is the transpose of a non-singular Vandermonde matrix and we have
\be
\big(\Pi^{-1}\big)_{k4} = \frac{\iu}{\pchar^{\,\prime}(\mu_k)} = \frac{\iu}{\mu_k(\mu_k^2+\ha)}.
\label{eq:vandermonde}
\ee
Let $D(x)$ be the $4\times 4$ diagonal matrix with matrix elements $D(x)_{jk}=\exp(\iu\mu_jx)\delta_{jk}$. Notice that $\Pi D(x)$ is a fundamental matrix of the system $y^\prime=Ay$ if $\lambda>\hp$. The $\lambda$ dependence of $\mu_k$, $p_k$, $\Pi$, and $D(x)$ is not always  explicitly shown to avoid clumsy expressions.

Theorem \ref{thm:sols_edo} is a corollary of the following two theorems about the solutions of the linear system (\ref{eq:system}).

\begin{thm}
\label{thm:sols_system_no_branch}
If $\lambda_s>\hp$ is large enough, there are functions $y_k^+$ and $y_k^-$, $1\leq k\leq 4$, with domain $\R\times[\lambda_s,\infty)$ and image in $\C^4$, which are written in the form
\be
\begin{gathered}
y_k^+(x,\lambda) = \exp(\iu\mu_kx) \Pi\big(e_k+v^+_k(x,\lambda)\big), \\[4pt]
 y_k^-(x,\lambda) = \exp(\iu\mu_kx) \Pi\big(e_k+v^-_k(x,\lambda)\big),
\end{gathered}
\label{eq:app:ykp_ykm}
\ee
for $(x,\lambda)\in\R\times[\lambda_s,\infty)$, such that
\begin{enumerate}
\item $\big\{y_k^+(\cdot,\lambda),\,1\leq k\leq 4\big\}$ and
$\big\{y_k^-(\cdot,\lambda),\,1\leq k\leq 4\big\}$ are two fundamental sets of solutions of the system (\ref{eq:system}) for each $\lambda\in[\lambda_s,\infty)$.
\item The functions $v_k^+$ and $v_k^-$ are continuous on $\R\times[\lambda_s,\infty)$, and satisfy
\bea
&&|v_k^+(x,\lambda)|\leq c\,\nu(\lambda)^{-1}\big(1+x^3\big)^{-1}, \quad x\geq 0, \quad \lambda\in[\lambda_s,\infty), \label{eq:app:bound_v} \\[4pt]
&&|v_k^-(x,\lambda)|\leq c\,\nu(\lambda)^{-1}\big(1-x^3\big)^{-1}, \quad x\leq 0, \quad \lambda\in[\lambda_s,\infty),
\hspace*{1.5cm}
\eea
where $c$ is a constant.
\item For each fixed $x\in\R$, $v_k^+(x,\cdot)$ and $v_k^-(x,\cdot)$ are analytic in $[\lambda_s,\infty)$.
\end{enumerate}
\end{thm}
\begin{proof}
For $1\leq k\leq 4$ we obtain $y_k^+$ as follows. For $1\leq j\leq 4$ define $s_{1j}=1$ if $\Imag\,\mu_j-\Imag\,\mu_k>0$ and $s_{1j}=0$ otherwise, and $s_{2j}=1-s_{1j}$. For $l=1,2$, let $D_l(x)$ be the $4\times 4$ diagonal matrix whose matrix elements are $D_l(x)_{ij} = s_{lj}\exp(\iu\mu_jx)\delta_{ij}$. Notice that $\Pi D_l(x)$ are solution matrices of the system (\ref{eq:system}), and $D(x)=D_1(x)+D_2(x)$. For $(x,\lambda)\in[0,\infty)\times[\lambda_s,\infty)$ consider the integral equation
\be
\begin{gathered}
y(x,\lambda) = \exp(\iu\mu_kx)p_k+\int_0^x \Pi D_1(x)D(-\tau)\Pi^{-1}B(\tau)y(\tau,\lambda)d\tau \\[4pt]
-\int_x^\infty \Pi D_2(x)D(-\tau)\Pi^{-1}B(\tau)y(\tau,\lambda)d\tau.
\end{gathered}
\ee
If this equation has a solution, then it is a solution of the system (\ref{eq:system}). The function $\tilde{y}(x,\lambda)=\Pi^{-1}(\lambda)y(x,\lambda)$, for $(x,\lambda)\in[0,\infty)\times[\lambda_s,\infty)$, satisfies the equation
\be
\begin{gathered}
\tilde{y}(x,\lambda) = \exp(\iu\mu_kx)e_k+\int_0^x D_1(x)D(-\tau)\tilde{B}(\tau)\tilde{y}(\tau,\lambda)d\tau \\[4pt]
-\int_x^\infty D_2(x)D(-\tau)\tilde{B}(\tau)\tilde{y}(\tau,\lambda)d\tau,
\end{gathered}
\label{eq:app:inteq2}
\ee
where $\tilde{B}(\tau)=\Pi B(\tau) \Pi^{-1}$. We need a bound on the norm of $\tilde{B}$. From equations (\ref{eq:AB}) and (\ref{eq:vandermonde}) we get for $1\leq i,j\leq 4$:
\be
\tilde{B}_{ij}(\tau) = \frac{\iu}{\pchar^{\,\prime}(\mu_i)}\sum_{l=1}^3\big(\iu\mu_j\big)^{l-1}B_{4l}(\tau).
\ee
Since $|\mu_j|^{l-1}/\pchar^{\,\prime}(\mu_i) \leq \theta^{l-1}/\big(2\nu(\theta^2+\nu^2)\big)\leq \theta^{l-3}/(2\nu)$
and $\theta\geq\theta_0>0$, we obtain $|\tilde{B}(\tau)|\leq (c_0/\nu)|B(\tau)|$, where $c_0$ is a constant. Let us try successive approximations to the solution of equation (\ref{eq:app:inteq2}): $\tilde{y}^{(0)}(x,\lambda)=0$ and for $j\in\N$
\be
\begin{gathered}
\tilde{y}^{(j)}(x,\lambda) = \exp(\iu\mu_kx)e_k+\int_0^x D_1(x)D(-\tau)\tilde{B}(\tau)\,\tilde{y}^{(j-1)}(\tau,\lambda)d\tau \\[4pt]
-\int_x^\infty D_2(x)D(-\tau)\tilde{B}(\tau)\,\tilde{y}^{(j-1)}(\tau,\lambda)d\tau.
\end{gathered}
\label{eq:app:approx}
\ee
We have $\tilde{y}^{(1)}(x,\lambda)=\exp(\iu\mu_k x)e_k$ and $|\tilde{y}^{(1)}(x,\lambda)|=\exp(-\Imag\,\mu_k x)$. The functions $\tilde{y}^{(j)}$ are continuous on $[0,\infty)\times[\lambda_s,\infty)$, and, for fixed $x$, $\tilde{y}^{(j)}(x,\cdot)$ is analytic in $[\lambda_s,\infty)$. Now we show that if $x\geq 0$ and $\lambda\geq\lambda_s$, with $\lambda_s$ large enough, then 
\be
\big|\tilde{y}^{(j+1)}(x,\lambda)-\tilde{y}^{(j)}(x,\lambda)\big|\leq (1/2)^j\exp\big(-\Imag\,\mu_k x\big)
\label{eq:app:ineq}
\ee
for $j\in\N$. We proceed by induction. For $j=0$ the inequality holds. Assuming that it holds for some $j\in\N\cup\{0\}$, equation (\ref{eq:app:approx}) imply
\be
\big|\tilde{y}^{(j+2)}(x,\lambda)-\tilde{y}^{(j+1)}(x,\lambda)\big|\leq\left(1/2\right)^j 
\exp(-\Imag\,\mu_kx)
\frac{c_0}{\nu(\lambda)}\int_0^\infty |B(\tau)|d\tau.
\ee
If $\lambda_s$ is large enough so that $c_0\,\nu(\lambda_s)^{-1}\int_0^\infty |B(\tau)|d\tau\leq 1/2$, then inequality (\ref{eq:app:ineq}) is proved (recall that $\nu$ is an increasing function of $\lambda$). Hence, the sequence $\tilde{y}^{(j)}$ converges uniformly on each compact set $[0,x_1]\times[\lambda_s,\lambda_1]$, with $x_1\in\R^+$ and $\lambda_1>\lambda_s$, to a function $\tilde{y}_k^+$ which, consequently, is continuous in $[0,\infty)\times[\lambda_s,\infty)$. Moreover, for fixed $x\in[0,\infty)$, the function $\tilde{y}_k^+(x,\cdot)$ is analytic in a neighborhood of $[\lambda_s,\infty)$. The bound $|\tilde{y}_k^+(x,\lambda)|\leq 2\exp(-\Imag\,\mu_kx)$ is easily obtained by summing the inequality
\be
|\tilde{y}^{(j+1)}(x,\lambda)| - |\tilde{y}^{(j)}(x,\lambda)| \leq (1/2)^j\exp(-\Imag\,\mu_kx),
\ee
from $j=0$ to $l$ and taking the limit $l\to\infty$. From $\tilde{y}_k^+(x,\lambda)$ we obtain $y_k^+(x,\lambda)=\Pi\tilde{y}_k^+(x,\lambda)$, and $v_k^+(x,\lambda)=\exp(-\iu\mu_kx)\tilde{y}_k^+(x,\lambda)$, so that
\be
\begin{gathered}
v_k^+(x,\lambda) = \exp(-\iu\mu_kx)
\left(\int_0^xD_1(x)D(-\tau)\tilde{B}(\tau)\tilde{y}_k^+(\tau,\lambda)d\tau\right. 
\\[4pt]
\left. -\int_x^\infty D_2(x)D(-\tau)\tilde{B}(\tau)\tilde{y}_k^+(\tau,\lambda)d\tau\right).
\end{gathered}
\ee
From this equation we get $|v_k^+(x,\lambda)|\leq (2c_0/\nu)\big(I_1(x)+I_2(x)\big)$,
where
\bea
I_1(x) = \int_0^x \sum_{j=1}^4s_{1j}\exp\Big(-(\Imag\mu_j-\Imag\mu_k)(x-\tau)\Big)|B(\tau)|d\tau, \label{eq:ineq_I1}\\[4pt]
I_2(x) = \int_x^\infty \sum_{j=1}^4s_{2j}\exp\Big(-(\Imag\mu_j-\Imag\mu_k)(x-\tau)\Big)|B(\tau)|d\tau. \label{eq:ineq_I2}
\eea
We have
\be
I_2(x) \leq 4\int_x^\infty |B(\tau)| d\tau \leq 4\int_x^\infty \frac{1+\tau^3}{1+x^3}|B(\tau)|d\tau\leq \frac{c_2}{1+x^3},
\ee
where $c_2$ is a constant, independent of $\lambda$. We also have
\be
\begin{gathered}
I_1(x)\leq 4 \int_0^x \exp\big(-\theta_0(x-\tau)\big)|B(\tau)| d\tau \leq \\[4pt] 
4\left(\int_0^{x/2} \exp\big(-\theta_0(x-\tau)\big)|B(\tau)| d\tau
+ \int_{x/2}^x |B(\tau)| d\tau\right) \leq \\[4pt] 
  4\left(\exp(-\theta_0x/2)\int_0^{x/2} |B(\tau)| d\tau
+ \int_{x/2}^x \frac{1+8\tau^3}{1+x^3}|B(\tau)| d\tau\right) \leq \frac{c_1}{1+x^3},
\end{gathered}
\ee
where $c_1$ is another constant. In this way we get the bound on $|v_k^+(x,\lambda)|$ of equation (\ref{eq:app:bound_v}). The function $y_k^+$ extended to $\R\times[\lambda_s,\infty)$ with the required  properties is obtained by solving the system (\ref{eq:system}) with the initial condition at $x=0$ given by $y(0)=\Pi(\lambda)\big(e_k+v_k^+(0,\lambda)\big)$.

The existence of the functions $y_k^-$ with the required properties is proved similarly.
\end{proof}

The matrix $\Pi(\lambda)$ is singular at $\lambda=\hp$ because in this case $\nu=0$ and $p_2=p_3$. A fundamental matrix of the system $y^\prime=Ay$ with an inverse bounded in $[\hp,\lambda_s]$ has to be built in a different way.
Let $\Psi(x)$ the $4\times 4$ matrix whose first, second, and fourth columns are $p_1$, $p_2$, and $p_4$, respectively, and whose third column is 
\be
\frac{1}{\nu}\big(p_3-\exp(-\iu 2\nu x)p_2\big) \text{\ \ if\ \ } \lambda\in(\hp,\lambda_s], \quad xe_1+e_2 \text{\ \ if\ \ } \lambda=\hp.
\ee
Again the $\lambda$ dependence of $\Psi(x)$ is not explicitly shown to avoid unwieldy expressions.
Clearly, $\Psi(x)D(x)$ is a fundamental matrix of $y^\prime=Ay$. We notice
\be
\det \Psi(x) =\frac{1}{\nu}\det\Pi = 2\theta(\theta^2+\nu^2)\geq 2\theta_0^3>0,\quad \lambda\geq\hp.
\ee
From this it is easy to see that there is $c_0$ is independent of $x$ and $\lambda$ such that $|\Psi(x)|\leq c_0 (1+|x|)$ and $|\Psi^{-1}(x)|\leq c_0 (1+|x|)$ hold for $(x,\lambda)\in\R\times[\hp,\lambda_s]$.
 
\begin{thm}
\label{thm:sols_system_branch}
For any $\lambda_s>\hp$ there are functions $z_k^+$ and $z_k^-$, $1\leq k\leq 4$, with domain $\R\times[\hp,\lambda_s]$ and image in $\C^4$, which are written in the form
\be
\begin{gathered}
z_k^+(x,\lambda) = \exp(\iu\mu_kx) \big(p_k+w^+_k(x,\lambda)\big), \\[4pt]
 z_k^-(x,\lambda) = \exp(\iu\mu_kx) \big(p_k+w^-_k(x,\lambda)\big),
\end{gathered}
\ee
such that
\begin{enumerate}
\item $\big\{z_k^+(\cdot,\lambda),\,1\leq k\leq 4\big\}$ and
$\big\{z_k^-(\cdot,\lambda),\,1\leq k\leq 4\big\}$ are two fundamental sets of solutions of the system (\ref{eq:system}) for each $\lambda\in(\hp,\lambda_s]$.
\item The functions $w_k^+$ and $w_k^-$ are continuous on $\R\times[\hp,\lambda_s]$, and satisfy
\bea
&&|w_k^+(x,\lambda)|\leq c\,(1+x)^{-1}, \quad x\geq 0, \quad \lambda\in[\hp,\lambda_s], \label{eq:app:bound_w} \\[4pt]
&&|w_k^-(x,\lambda)|\leq c\,(1-x)^{-1}, \quad x\leq 0, \quad \lambda\in[\hp,\lambda_s],
\hspace*{1.5cm}
\eea
where $c$ is a constant (independent of $x$ and $\lambda$, it may depend on $\lambda_s$).
\item For fixed $x\in\R$, $w_k^+(x,\cdot)$ and $w_k^-(x,\cdot)$ are analytic in $(\hp,\lambda_s]$.
\end{enumerate}
\end{thm}
\begin{proof}
For $1\leq k\leq 4$ we obtain $z_k^+$ as follows. For $l=1,2$ and $1\leq j\leq 4$
define $s_{lj}$ and $D_l(x)$ as in the proof of theorem \ref{thm:sols_system_branch}.
Notice that $D(x)=D_1(x)+D_2(x)$ and that $\Psi(x) D_l(x)$, $l=1,2$, are solution matrices of the system (\ref{eq:system}). Let $x_0\geq 0$ be a constant to be determined, and for $(x,\lambda)\in[x_0,\infty)\times[\hp,\lambda_s]$ consider the integral equation
\be
\begin{gathered}
y(x,\lambda) = \exp(\iu\mu_kx)p_k+\int_{x_0}^x \Psi(x) D_1(x)D(-\tau)\Psi^{-1}(\tau)B(\tau)y(\tau,\lambda)d\tau \\[4pt]
-\int_x^\infty \Psi(x) D_2(x)D(-\tau)\Psi^{-1}(\tau)B(\tau)y(\tau,\lambda)d\tau.
\end{gathered}
\ee
If this equation has a solution, then it is a solution of the system (\ref{eq:system}). Let us consider the successive approximations $y^{(0)}(x,\lambda)=0$ and, for $j\in\N$,
\be
\begin{gathered}
y^{(j)}(x,\lambda) \!=\! \exp(\iu\mu_kx)p_k \!+\! \!\int_{x_0}^x\!\!\!\! \Psi(x) D_1(x)D(-\tau)\Psi^{-1}(\tau)B(\tau)\,y^{(j-1)}(\tau,\lambda)d\tau \\[4pt]
-\int_x^\infty \Psi(x)D_2(x)D(-\tau)\Psi^{-1}(\tau)B(\tau)\,y^{(j-1)}(\tau,\lambda)d\tau.
\end{gathered}
\label{eq:app:approx_2}
\ee
We have $y^{(1)}(x,\lambda)=\exp(\iu\mu_k x)p_k$ and $|y^{(1)}(x,\lambda)|=\exp(-\Imag\,\mu_k x)|p_k|$. 
The functions $y^{(j)}$ are continuous on $[x_0,\infty)\times[\hp,\lambda_s]$ and, for fixed $x\geq x_0$, $y^{(j)}(x,\cdot)$ are analytic in $(\hp,\lambda_s]$.
Define
$c_1=\max\{|p_k|,\,1\leq k\leq 4,\,\hp\leq\lambda\leq\lambda_s\}$.
We now show that if $x_0$ is large enough then
\be
\big|y^{(j+1)}(x,\lambda)-y^{(j)}(x,\lambda)\big|\leq c_1 (1/2)^j \exp(-\Imag\,\mu_k x)
\label{eq:app:ineq_2}
\ee
for  $x\geq x_0$, $\lambda\in[\hp,\lambda_s]$, and $j\in\N\cup\{0\}$. We proceed by induction. For $j=0$ the inequality holds. Suppose that it holds for some $j\in\N\cup\{0\}$. This assumption, equation (\ref{eq:app:approx_2}), and the bounds on $\Psi(x)$ and $\Psi^{-1}(x)$ imply
\be
\big|y^{(j+2)}(x,\lambda)-y^{(j+1)}(x,\lambda)\big|\leq c_2(1+x)(1/2)^j \exp(-\Imag\,\mu_kx) J(x),
\ee
where $c_2$ is a constant, $J(x) = I_1(x)+I_2(x)$, and
\bea
I_1(x) \!=\!\! \int_{x_0}^x \!\sum_{j=1}^4s_{1j}\exp\Big(\!-(\Imag\mu_j-\Imag\mu_k)(x-\tau)\Big)(1+\tau)|B(\tau)|d\tau, 
\label{eq:app:I1} \\[4pt]
I_2(x) \!=\!\! \int_x^\infty \!\sum_{j=1}^4s_{2j}\exp\Big(\!-(\Imag\mu_j-\Imag\mu_k)(x-\tau)\Big)(1+\tau)|B(\tau)|d\tau.
\label{eq:app:I2}
\eea
As in the proof of theorem \ref{thm:sols_system_no_branch}, inequalities (\ref{eq:ineq_I1}) and (\ref{eq:ineq_I2}), we get that $I_1(x)$ and $I_2(x)$ are bounded by $c_3/(1+x)^2$ if $x\geq x_0$.
Then, choosing $x_0$ so that $2c_3c_2/(1+x_0)\leq 1/2$ we get
\be
\big|y^{(j+2)}(x,\lambda)-y^{(j+1)}(x,\lambda)\big|\leq c_1(1/2)^{j+1} \exp(-\Imag\,\mu_kx),
\ee
for all $x\geq x_0$ and $\lambda\in[\hp,\lambda_s]$, and inequality (\ref{eq:app:ineq_2}) is proved. Hence, the sequence $y^{(j)}$ converges uniformly on each set $[x_0,x_1]\times[\hp,\lambda_s]$, with $x_1>x_0$, to a function $z_k^+$ which, consequently, is continuous on $[x_0,\infty)\times[\hp,\lambda_s]$. Moreover, for fixed $x\in[x_0,\infty)$, the function $z_k^+(x,\cdot)$ is analytic on $(\hp,\lambda_s]$. The bound $|z_k^+(x,\lambda)|\leq 2c_1\exp(-\Imag\,\mu_kx)$, for $x\geq x_0$, is easily obtained by summing the inequality
\be
|y^{(j+1)}(x,\lambda)| - |y^{(j)}(x,\lambda)| \leq c_1(1/2)^j\exp(-\Imag\,\mu_kx),
\ee
from $j=0$ to $l$ and taking the limit $l\to\infty$. From $z_k^+(x,\lambda)$ we obtain $w_k^+(x,\lambda)=\exp(-\iu\mu_kx)z_k^+(x,\lambda)-p_k$, so that
\be
\begin{gathered}
\!w_k^+(x,\lambda) = \exp(-\iu\mu_kx)\!\left(\int_{x_0}^x\!\!\Psi(x)D_1(x)D(-\tau)\Psi^{-1}(\tau)B(\tau)z_k^+(\tau,\lambda)d\tau\right. 
\\[4pt]
\left. -\int_x^\infty \Psi(x)D_2(x)D(-\tau)\Psi^{-1}(\tau)B(\tau)z_k^+(\tau,\lambda)d\tau\right).
\end{gathered}
\ee
From this equation and the bounds on the different factors of the integrands, we get $|w_k^+(x,\lambda)|\leq 2c_1c_0^2(1+x)\exp(-\Imag\,\mu_kx)\big(I_1(x)+I_2(x)\big)$, for $x\geq x_0$,
where $I_1(x)$ and $I_2(x)$ are given by equations (\ref{eq:app:I1}) and (\ref{eq:app:I2}), respectively.
From the bounds on $I_1(x)$ and $I_2(x)$ obtained above we get the estimate $|w_k^+(x,\lambda)|\leq c/(1+x)$ for $(x,\lambda)\in[x_0,\infty)\times[\hp,\lambda_s]$. The function $z_k^+$ extended to $\R\times[\hp,\lambda_s]$ with the required  properties is obtained by solving the system (\ref{eq:system}) with the initial condition at $x=x_0$ given by $y(x_0)=p_k+w_k^+(x_0,\lambda)$. The functions $z^+_k$ thus obtained are continuous $\R\times[\hp,\lambda_s]$ and hence bounded on $[0,x_0]\times[\hp,\lambda_s]$. Then, the bound on $w_k^+$ can be extended to $[0,\infty)\times[\hp,\lambda_s]$, by increasing $c$ if necessary.

The existence of the functions $z_k^-$ with the required properties is proved similarly.
\end{proof}

For $1\leq k\leq 4$ the functions $\phip_k$ and $\phim_k$ of theorem \ref{thm:sols_edo} are obtained from the functions $y_k^+$, $y_k^-$, $z_k^+$, and $z_k^-$ of theorems \ref{thm:sols_system_no_branch} and \ref{thm:sols_system_branch} as follows. For $x\in\R$:
\be
\begin{gathered}
\phip_k(x,\lambda)=\left\{
\begin{array}{ll}
e_1\cdot z_k^+(x,\lambda), & \lambda\in[\hp,\lambda_s], \\[4pt]
e_1\cdot y_k^+(x,\lambda), & \lambda\in(\lambda_s,\infty), 
\end{array}
\right.
\\[4pt]
\phim_k(x,\lambda)=\left\{
\begin{array}{ll}
e_1\cdot z_k^-(x,\lambda), &  \lambda\in[\hp,\lambda_s], \\[4pt]
e_1\cdot y_k^-(x,\lambda), & \lambda\in(\lambda_s,\infty).
\end{array}
\right.
\end{gathered}
\label{eq:app:phip_phim}
\ee

\section{Proof of an ancillary result \label{app:ancillary}}

The next lemma is used in the proof of lemma \ref{thm:sols_bounded}.

\begin{lem}
\label{thm:continuity_coefficients}
Let $n$ be a natural number, $I$ and $M$ subsets of $\R$, with $I$ open, and $f_j:I\times M\to\C$, $1\leq j\leq n$, continuous bounded functions. Suppose that for each $\lambda\in M$ the functions $f_j(\cdot,\lambda)$ are linearly independent, consider functions $c_j:M\to\C$, and define $f:I\times M\to\C$ by
\be
f(x,\lambda) = \sum_{j=1}^nc_j(\lambda)f_j(x,\lambda), \quad  (x,\lambda)\in I\times M.
\ee
If $f$ is continuous and bounded, then the functions $c_j$ are continuous.
\end{lem}
\begin{proof}
The functions $c_j$ are bounded on any compact subset $J\subseteq M$. Otherwise, there is a sequence $\lambda_m\in J$ such that $\lim_m a_m=\infty$, where $a_m=\max\{|c_j(\lambda_m)|,\,1\leq j\leq n\}$. Since $J$ is compact, the sequence can be taken convergent, so that $\lambda_m\to\lambda_0\in J$. Obviously, we have 
$|c_j(\lambda_m)|/a_m\leq 1$ and $\max\{ |c_j(\lambda_m)|/a_m,\,1\leq j\leq n\} = 1$, and therefore the sequence $\lambda_m$ can be taken so that $|c_j(\lambda_m)|/a_m$ is convergent for $1\leq j\leq n$. Let $\alpha_j$ denote the corresponding limit. We have
\be
\sum_{j=1}^n \frac{|c_j(\lambda_m)|}{a_m}f_j(x,\lambda_m) = \frac{1}{a_m}f(x,\lambda_m).
\ee
Taking the limit $m\to\infty$ in the above equality we obtain $\sum_j\alpha_j f_j(x,\lambda_0)=0$, and the linear independence of the functions $f_j(\cdot,\lambda_0)$ implies $\alpha_j=0$ for $1\leq j\leq n$. However, this is not possible since $\max\{ |c_j(\lambda_m)|/a_m,\,1\leq j\leq n\} = 1$ for all $m\in\N$. Therefore, the functions $c_j$ are bounded on each compact set $J$ contained in $M$.

Now take a sequence $\lambda_m\in M$ which converges to $\lambda\in M$. The set $\{\big(c_1(\lambda_m),\ldots,c_n(\lambda_m)\big),\,m\in\N\}$ is bounded in $\C^n$. Let $(a_1,\ldots,a_n)\in\C^n$ be one of its accumulation points. There is a subsequence $\lambda_m^\prime$ such that $c_j(\lambda_m^\prime)$ converges to $a_j$ for $1\leq j\leq n$. Taking the limit $m\to\infty$ in the expression 
\be
\begin{gathered}
\sum_{j=1}^n\big(c_j(\lambda)-c_j(\lambda_m^\prime)\big)f_j(x,\lambda) = f(x,\lambda) - f(x,\lambda_m^\prime) \\[-6pt]+ \sum_{j=1}^nc_j(\lambda_m^\prime)\big(f_j(x,\lambda)-f_j(x,\lambda_m^\prime)\big)
\end{gathered}
\ee
we get $\sum_{j=1}^n\big(c_j(\lambda)-a_j\big)f_j(x,\lambda)=0$ for all $x\in\R$.
The linear independence of the functions $f_j(\cdot,\lambda)$ implies that $a_j=c_j(\lambda)$. Therefore, the accumulation point of the sequence $c_j(\lambda_m)$ is unique and the same for all sequences $\lambda_m$ that converge to $\lambda$, and is given by $c_j(\lambda)$. Hence the functions $c_j$ are continuous on $M$.
\end{proof}

\bibliographystyle{apa}
\bibliography{references_spectral.bib}

\begin{thebibliography}{}

\bibitem[\protect\astroncite{Bade}{1954}]{Bade1954}
Bade, W.~G. (1954).
\newblock Unbounded spectral operators.
\newblock {\em Pacific J. Math}, 4:373--392.

\bibitem[\protect\astroncite{Bagarello et~al.}{2015}]{Bagarello2015}
Bagarello, F., Gazeau, J.-P., Szafraniec, F.~H., and Znojil, M., editors
  (2015).
\newblock {\em Non-Selfadjoint Operators in Quantum Physics. Mathematical
  Aspects.}
\newblock {John Wiley \& Sons}.

\bibitem[\protect\astroncite{Coddington and Levinson}{1972}]{Coddington1972}
Coddington, E.~A. and Levinson, N. (1972).
\newblock {\em Theory of Ordinary Differential Equations}.
\newblock Tata McGraw-Hill.

\bibitem[\protect\astroncite{Dunford}{1954}]{Dunford1954}
Dunford, N. (1954).
\newblock Spectral operators.
\newblock {\em Pac. J. Math.}, 4:321--354.

\bibitem[\protect\astroncite{Dunford}{1958}]{Dunford1958}
Dunford, N. (1958).
\newblock A survey of the theory of spectral operators.
\newblock {\em Bull. Amer. Math. Soc.}, 64:217--274.

\bibitem[\protect\astroncite{Dunford and Schwartz}{1971}]{Dunford1971}
Dunford, N. and Schwartz, J.~T. (1971).
\newblock {\em Linear operators, Part III, Spectral operators}.
\newblock {John Wiley \& Sons}.

\bibitem[\protect\astroncite{Edmunds and Evans}{2018}]{Edmunds2018}
Edmunds, D. and Evans, W. (2018).
\newblock {\em Spectral Theory and Differential Operators}.
\newblock Oxford University Press.

\bibitem[\protect\astroncite{Huige}{1971}]{Huige1971}
Huige, G.~E. (1971).
\newblock Perturbation theory of some spectral operators.
\newblock {\em Comm. Pure and Appl. Math.}, 24(6):741 – 757.

\bibitem[\protect\astroncite{Kemp}{1960}]{Kemp1960}
Kemp, R. R.~D. (1960).
\newblock On a class of non-self-adjoint differential operators.
\newblock {\em Canadian Journal of Mathematics}, 12:641--659.

\bibitem[\protect\astroncite{Kramer}{1957}]{Kramer1957}
Kramer, H.~P. (1957).
\newblock {Perturbations of differential operators}.
\newblock {\em Pacific J. Math.}, 7:1405--1435.

\bibitem[\protect\astroncite{Laliena and Campo}{2021}]{Laliena2021}
Laliena, V. and Campo, J. (2021).
\newblock {Magnonic Goos–H\" anchen effect induced by 1D solitons}.
\newblock {\em Adv. Electron. Mater.}, 2100782.

\bibitem[\protect\astroncite{Mackey}{1952}]{Mackey1952}
Mackey, G.~W. (1952).
\newblock {\em Commutative Banach algebras}.
\newblock {Cambridge, Mass.: Harvard University}.

\bibitem[\protect\astroncite{McGarvey}{1965}]{McGarvey1965}
McGarvey, D.~C. (1965).
\newblock {Operators commuting with translations by one. Part III. Perturbation
  results for periodic differential operators.}
\newblock {\em J. Math. Anal. Appl.}, 12:187--234.

\bibitem[\protect\astroncite{Schwartz}{1960}]{Schwartz1960}
Schwartz, J. (1960).
\newblock {Some non‐selfadjoint operators}.
\newblock {\em Comm. Pure Appl. Math.}, 13(4):609 – 639.

\bibitem[\protect\astroncite{Schwartz}{1954}]{Schwartz1954}
Schwartz, J.~T. (1954).
\newblock {Perturbations of spectral operators, and applications. I. Bounded
  perturbations.}
\newblock {\em Pacific J. Math.}, 4:415--458.

\bibitem[\protect\astroncite{Teschl}{2009}]{Teschl2009}
Teschl, G. (2009).
\newblock {\em Mathematical Methods in Quantum Mechanics}.
\newblock American Mathematical Society.

\bibitem[\protect\astroncite{Turner}{1966}]{Turner1966}
Turner, R. E.~L. (1966).
\newblock {Perturbation of ordinary differential operators.}
\newblock {\em J. Math. Anal. Appl.}, 13:447--457.

\bibitem[\protect\astroncite{Wermer}{1954}]{Wermer1954}
Wermer, J. (1954).
\newblock {Commuting spectral measures on Hilbert space}.
\newblock {\em Pacific J. Math.}, 4:355--361.

\end{thebibliography}

\end{document}